\documentclass{amsart}
\usepackage{tikz}
\usepackage{tikz-cd} 
\usepackage{pgfplots}
\pgfplotsset{width=10cm,compat=1.9}
\usepackage[export]{adjustbox}
\usepackage{caption}
\usepackage{subcaption}
\usepackage{wrapfig}
\usepackage{graphicx}
\usepackage{mathtools}  
\usepackage{relsize} 

\usepackage{fdsymbol} 

\usepackage{amsmath,systeme} 

\usepackage{courier}

\usepackage[all,cmtip]{xy}

\usepackage{enumerate}

\usepackage{mathrsfs} 

\usepackage[mathscr]{euscript}

\usepackage[OT2,T1]{fontenc}
\DeclareSymbolFont{cyrletters}{OT2}{wncyr}{m}{n}
\DeclareMathSymbol{\Sha}{\mathalpha}{cyrletters}{"58}

\theoremstyle{plain}
\newtheorem{Theorem}[subsection]{Theorem}
\newtheorem{Proposition}[subsection]{Proposition}
\newtheorem{Lemma}[subsection]{Lemma}
\newtheorem{Corollary}[subsection]{Corollary}

\theoremstyle{definition}
\theoremstyle{remark}
\newtheorem{Definition}[subsection]{\bf Definition}
\newtheorem{Example}[subsubsection]{\bf Example}

\newtheorem{Remark}[subsection]{\bf Remark}

\numberwithin{equation}{subsection}

\renewcommand{\b}{{\boldsymbol b}}

\renewcommand{\v}{{\mathtt v}}

\newcommand{\cB}{{\text{\rm B}}}

\newcommand{\cC}{{\text{\rm C}}}
\newcommand{\F}{{\mathbb F}}
\newcommand{\G}{{\text{\rm G}}}

\renewcommand{\H}{{\text{\rm H}}}

\newcommand{\M}{{\text{\rm M}}}
\newcommand{\N}{{\mathbb N}}

\newcommand{\Q}{{\mathbb Q}}

\newcommand{\X}{{\text{\rm X}}}

\newcommand{\cZ}{{\text{\rm Z}}}
\newcommand{\wZ}{{\widehat{\mathbb Z}}}
\newcommand{\Z}{{\mathbb Z}}

\newcommand{\alt}{{\text{\rm alt}}}
\newcommand{\bmx}[1]{{\begin{bmatrix}#1\end{bmatrix}}}
\newcommand{\bmxr}[1]{{\begin{bmatrix*}[r]#1\end{bmatrix*}}}

\newcommand{\br}[1]{{\left<{#1}\right>}}

\newcommand{\car}{{\text{\rm char}}}
\newcommand{\ccup}{{\;\mathsmaller\cup\;}}

\newcommand{\col}{{\text{\rm col}}}

\newcommand{\ds}[1]{{\displaystyle{#1}}}

\renewcommand{\gcd}{{\text{\rm gcd}}}

\newcommand{\id}{{\text{\rm id}}}

\newcommand{\ind}{{\text{\rm ind}}}
\renewcommand{\inf}{{\text{\rm inf}\,}}

\newcommand{\inv}{{\text{\rm inv}}}
\newcommand{\isim}{{\;\overset{\sim}{\longrightarrow}\;}}
\newcommand{\isom}{{\;\simeq\;}}
\renewcommand{\ker}{{\text{\rm ker}}}
\newcommand{\lcm}{{\text{\rm lcm}}}

\newcommand{\lr}{\longrightarrow}

\newcommand{\nr}{{\text{\rm nr}}}

\newcommand{\overbar}[1]{\mkern 1.5mu\overline{\mkern-1.5mu#1\mkern-1.5mu}\mkern 1.5mu}

\newcommand{\per}{{\text{\rm per}}}

\newcommand{\pfaff}{{\text{\rm pfaff}}}
\renewcommand{\pmod}[1]{{\text{\rm (mod ${#1}$)}}}
\newcommand{\prestar}{{{}^*\!}}

\newcommand{\res}{{\text{\rm res}}}

\newcommand{\row}{{\text{\rm row}}}
\newcommand{\sep}{{\text{\rm sep}}}

\newcommand{\sym}{{\text{\rm sym}}}
\newcommand{\tg}{{\text{\rm tg}}}

\newcommand{\ttr}{{\text{\rm tr}}}
\newcommand{\tr}{{\text{\rm t}}}

\newcommand{\un}{\underline}

\newcommand{\Alt}{{\text{\rm Alt}}}

\newcommand{\Aut}{{\text{\rm Aut}}}

\newcommand{\Br}{{\text{\rm Br}}}

\newcommand{\Ext}{{\text{\rm Ext}}}

\newcommand{\Gal}{{\text{\rm Gal}}}

\newcommand{\Hom}{{\text{\rm Hom}}}

\newcommand{\Nm}{{\text{\rm N}}}

\begin{document}
\small

\title[Tame Brauer Group of a Higher Local Field]
{Hasse Invariant for the Tame Brauer Group of a Higher Local Field}


\author{Eric Brussel}
\address{
Department of Mathematics\\
California Polytechnic State University}
\email{ebrussel@calpoly.edu}

\begin{abstract}
We generalize the Hasse invariant of local class field theory 
to the tame Brauer group of a higher dimensional local field,
and use it to study the arithmetic of central simple algebras, which
are given {\it a priori} as tensor products of standard cyclic algebras. 

We also compute the tame Brauer dimension (or {\it period-index bound}) and 
the cyclic length of a general henselian-valued field of finite rank and finite residue field.
\end{abstract}

\maketitle

\section{Introduction}

One of the seminal works of 20th century number theory was the 1932 paper \cite{BHN32}
by Brauer, Hasse, and Noether. 
The main theorems, cyclicity and period-index for division algebras over $\Q$, formed
a basis for the development of class field theory, and
were in turn based on Hasse's local class field theory, in Hasse's 1931 
article \cite{Hasse31}, see \cite[6.1]{Roq05}. 
Key to the latter was a natural isomorphism, the {\it Hasse invariant map}
\[\inv_F:\H^2(F,\mu_n)\isim\tfrac 1n\Z/\Z\]
over a local field $F$.
This map summarizes the Brauer group over a local field, providing
not only a functorial group isomorphism, but also the index of a class, 
and a structure theorem for $F$-division algebras, showing them all to be cyclic.

The Hasse invariant was generalized to higher dimensional local fields by Kato
in his higher local class field theory \cite[Theorem 3]{Ka79b}.
Kato established a natural isomorphism
\[\inv_F:\H^{d+1}(F,\mu_n^{\otimes(d-1)})\isim\tfrac 1n\Z/\Z\]
over a local field $F$ of dimension $d$, and used it to extend
the classical reciprocity map to higher local class field theory. 

Though a cornerstone of class field theory, Hasse's original motivation 
in \cite{Hasse31} was to understand the arithmetic of central simple algebras over local fields
\cite[Remark 6.3.1]{Roq05}. To extend this aspect of his work, 
we need to work in the Brauer group.

In this paper we generalize Hasse's invariant map to the prime-to-$p$ part
of $2\,\Br(F)$ for a $d$-dimensional henselian-valued field $F$ with finite residue field $k$.
Our generalization contains a precise index formula; a structure theorem
for division algebras, showing them to be tensor products of cyclic division algebras;
and an algorithm for expressing the underlying division algebra of a given central simple
algebra as a tensor product of cyclic division algebras.

We also compute the Brauer dimension, cyclic length, and period-index ratio
for the tame Brauer group. Though these results are expected,
and easily proved for higher local fields,
they are new for general henselian-valued fields with finite residue field.

Our generalization of the Hasse invariant 
is a map from the $n$-torsion ${}_n\Br(F)$ into the group of skew-symmetric matrices:
\[\H^2(F,\mu_n)\lr\Alt_{d+1}(\tfrac 1n\Z/\Z)\,,\]
When restricted to ${}_n2\,\Br(F)$ it factors through a natural isomorphism with 
the wedge product ${}_n2\,\X(F)\wedge\cZ^1(F,\mu_n)$,
and it is functorial with respect to basis change,
so that $2$-block diagonalization on the right ``diagonalizes'' a Brauer class.
We prove the index of the Brauer class is the order of the pfaffian subgroup
of the skew-symmetric matrix, and, like the determinant, has an explicit 
formula in terms of the matrix coefficients.
At the same time we show the underlying 
division algebra is a tensor product of cyclic division algebras of degrees equal to 
the abelian group theory invariants of the row space of the skew-symmetric matrix.
To prove our index formula we use the valuation theoretic framework of
\cite{TW15}. Our methods thus generalize 
\cite[Example 1.2.8]{TW15}, which was the method used by Hasse in \cite{Hasse31} to compute
the Hasse invariant over an ordinary local field.
 
For example, if $F=\F_5((t_1))((t_2))((t_3))$ is the $3$-dimensional local field, 
and we put $t_0=[2]_5$, then the sum of quaternions
\[\alpha=(t_0,t_1)+(t_0,t_2)+(t_0,t_3)+(t_1,t_2)+(t_1,t_3)+(t_2,t_3)\]
is mapped to the skew-symmetrix matrix
\[\alt_{\un t}(\alpha)=\bmx{0&\tfrac 12&\tfrac 12&\tfrac 12\\[2pt]
-\tfrac 12&0&\tfrac 12&\tfrac 12\\[2pt]
-\tfrac 12&-\tfrac 12&0&\tfrac 12\\[2pt]
-\tfrac 12&-\tfrac 12&-\tfrac 12&0}\]
The pfaffian formula computes
$a_{12}a_{34}-a_{13}a_{24}+a_{14}a_{23}=1/4$ in $\Q/\Z$, so the index is $4$.
This calculation is unexpectedly easy.
In this case the index is computable in principle from the Witt index formula for a Laurent series field,
but that approach becomes hopelessly complicated even for low values of $d$,
and does not extend to more general henselian fields.

Elements of $\Br(F)$ that do not fit into our theory include the $p$-primary part $\Br(F)_p$,
where $p=\car(k)$, and elements of $\Br(F)_2-2\,\Br(F)_2$.
The $p$-primary part has a much different character, and is not even
a candidate for a similar treatment. And though when $2\neq p$, $\Br(F)_2$ has the right number
of summands to define a map into skew-symmetric matrices, we prove that there is no natural map
until we restrict to the subgroup $2\,\Br(F)_2$. 
In particular, the {\it ad hoc} map is not functorial with respect to basis change, 
and is therefore useless. 
The problem with $2$-torsion could be anticipated, since totally ramified
classes $(t,t)_n$ violate ``skew-symmetry''  when $n$ is the maximal power of $2$ dividing $|\mu(F)|$.

In \cite{Br01b} and \cite{Br07} the author used the machinery of symplectic modules 
to derive an index formula and minimal expression for an arbitrary class in 
the (tame) Brauer group of a strictly henselian field, that is, a henselian field whose
residue field is algebraically closed.
The central observation was that the cocycles underlying a Brauer class define alternating
forms on the Galois group.
A more valuation-theoretic treatment using symplectic modules is formalized by 
Tignol and Wadsworth in \cite[Chapter 7]{TW15}.
These ideas predated the present author's work, and are rooted in (more general) work by Amitsur-Rowen-Tignol (\cite{ART}),
Tignol (\cite{T82}), Tignol-Amitsur (\cite{TA86}), and Tignol-Wadsorth (\cite{TW87}).
Their methods are especially powerful when the residue fields are algebraically closed,
and much is reduced to valuation theory. 
With non algebraically closed residue fields, 
the augmentation with Galois theoretic methods
seems necessary to obtain a comparable understanding.
There remains the problem of how to construct a group of Galois symplectic modules isomorphic to 
the tame Brauer group, as outlined in \cite[Chapter 7]{TW15}.

The paper is organized as follows. We first prove structure theorems for the character group
and Brauer group of a henselian-valued field of rank $d$, with finite residue field.
These results are well-known; we prove them for convenience, and use them immediately
to compute Brauer dimension and cyclic length. Then we prove Theorem~\ref{cocycleskim},
which shows ${}_n2\,\Br(F)$ is naturally
isomorphic to a wedge product, 
and we define basis and basis change for this wedge product, before proving the main
results, Theorem~\ref{alta} and Theorem~\ref{alt}, relating the arithmetic of the Brauer group
to the arithmetic of skew-symmetric matrices with coefficients in $\Q/\Z$.
We end by showing why the missing set ${}_n\Br(F)-{}_n2\,\Br(F)$ resists the same
treatment, even when $F$ contains $\mu_n$.

The author thanks Kelly McKinnie for discussions leading to the idea for this paper,
and Adrian Wadsworth's former student Frank Chang for alerting the author to an early hazard. 
The author dedicates this paper to his parents, Morton and Phyllis Brussel.

\section{Preliminaries}

Let $F$ be a field, let $n$ be a number invertible in $F$,
and let $m=|\mu_n(F)|$. 
Let $\Br(F)=\H^2(F,F_\sep^\times)$ and $\X(F)=\H^1(F,\Q/\Z)$ 
denote the Brauer group and character group of $F$, with $n$-torsion subgroups
${}_n\Br(F)=\H^2(F,\mu_n)$ and ${}_n\X(F)=\H^1(F,\tfrac 1n\Z/\Z)$.
If $\theta\in{}_n\X(F)$ has order $n$, let $F(\theta)/F$ denote the corresponding
cyclic extension of degree $n$. 

The coboundary map $\partial:F_\sep^\times\to\cC^1(\G_F,F_\sep^\times)$,
which takes an element $x\in F_\sep^\times$ to the function $\sigma\mapsto x^{\sigma-1}$,
has kernel $F^\times$ by Galois theory, and image $\cB^1(\G_F,F_\sep^\times)$, which
equals $\cZ^1(\G_F,F_\sep^\times)$ by Hilbert 90.
Since ${}_n(F_\sep^\times/F^\times)=F^{\times 1/n}/F^\times$ and ${}_n\cZ^1(\G_F,F_\sep^\times)=\cZ^1(F,\mu_n)$,  
we have a natural isomorphism 
\[\partial:F^{\times 1/n}/F^\times\isim\cZ^1(F,\mu_n)\]
We suppress the notation $\partial$, and write $z^{1/n}$ for 
the cocycle $\partial(z^{1/n })$.
The coboundaries $\cB^1(F,\mu_n)$ are the image $\partial(\mu_nF^\times/F^\times)\isom\mu_n/\mu_m$,
and we have a commutative diagram
\[\begin{tikzcd}
1\arrow[r]&\mu_n/\mu_m\arrow[r]\arrow[d, "\wr"]&F^{\times 1/n}/F^\times\arrow[r, "n"]\arrow[d, "\wr"]
&F^\times/F^{\times n}\arrow[r]\arrow[d, "\wr"]&1\\
0\arrow[r]&\cB^1(F,\mu_n)\arrow[r]&\cZ^1(F,\mu_n)\arrow[r]&\H^1(F,\mu_n)\arrow[r]&0
\end{tikzcd}
\]
Let $(t)_n\in\H^1(F,\mu_n)$ denote the class of $t^{1/n}$.

For a prime $p$ let $\mu'$ be the group of all prime-to-$p$ roots of unity,
let $(\Q/\Z)'$ be the prime-to-$p$ part of $\Q/\Z$,
and let \[\zeta^*:\mu'\,\isim(\Q/\Z)'\]
be a fixed isomorphism. Put $\zeta_n^*=\zeta^*|_{\mu_n}$
and set $\zeta_n=(\zeta_n^*)^{-1}(1/n)$.
Since $\mu_m\leq F^\times$, 
$\zeta_m^*$ is a $\G_F$-module isomorphism, and we have an induced isomorphism
\begin{align}
\label{zetam}
\begin{split}
\zeta_m^*:\H^1(F,\mu_m)&\isim\H^1(F,\tfrac 1m\Z/\Z)\\
(t)_m&\longmapsto(t)_m^*:=\zeta_m^*\circ t^{1/m}
\end{split}
\end{align}
If $\theta=(t)_m^*$ has order $m$,
then $F(\theta)=F(t^{1/m})$ has degree $m$ over $F$ (see \cite[XIV.2]{Se79}).

\subsection{\bf Cyclic Classes}\label{cyclic}
If $|\theta|=n$, the {\it cyclic Brauer class} $(\theta,t)\in\Br(F)$ is the cup 
product $\theta\ccup(t)_n$. 
The cyclic class $(\theta,t)$ determines the {\it cyclic algebra of degree $n$},
\[\Delta(\theta,t):=\{F(\theta)[y]:y^n=t,xy=y\sigma(x)\;\forall x\in F(\theta)\}\]
where $\sigma\in\Gal(F(\theta)/F)$ satisfies $\theta(\sigma)=1/n$.
Since the cup product is induced by the tensor product, 
if $m|n$ then $\theta\ccup(t)_m=\theta\ccup(t^{n/m})_n=(\tfrac nm\theta,t)$,
and $(s)_m^*\ccup(t)_n=(s)_m^*\ccup(t)_m=(s,t)_m$.
If $\theta=(s)_m^*$, then $(\theta,t)=(s,t)_m$, the {\it symbol Brauer class},
is represented by the {\it symbol algebra} 
\[\Delta(s,t)_m:=\{F[x,y]:x^m=s,y^m=t,[x,y]=\zeta_m\}\]
The {\it norm criterion} states that $(\theta,t)=0$ if and only if $t$ is a norm from $F(\theta)$.
If $\theta=(t)_m^*$ 
then $(\theta,t)=(t,t)_m=0$ if $m$ is odd, and
$(\theta,-t)=(t,-t)_m=0$ if $m$ is even, in which case
$(t,t)_m=(t,-1)_m$. 

Since $1$ is a norm, $(\theta,-1)$ has order dividing $2$.
If $\theta=2\theta'$ for some $\theta'$,
then $(\theta,-1)=2(\theta',-1)=0$. Conversely if $(\theta,-1)=0$ then
$\theta=2\theta'$ for some $\theta'$ by \cite{AFSS}.
Thus $(\theta,-1)=0$ if and only if $\theta\in 2\,\X(F).$ 
We call this result {\it Albert's criterion}.

\subsection{\bf 7-Term Sequence}\label{7term}
For an exact sequence $1\to N\to G\to\overbar{G}\to 1$ of profinite groups, with $N$ closed,
the Hochschild-Serre spectral sequence $\H^p(\overbar G,\H^q(N,M))\Rightarrow\H^{p+q}(G,M)$
yields a 7-term sequence (\cite[Appendix B]{M}), which we note for the reader's convenience:
\begin{align*}
0\lr &\H^1(\overbar G,M^{N})\lr \H^1(G,M)\lr 
\H^1(N,M)^{\overbar G}\lr \H^2(\overbar G,M^{N})\\
&\lr \ker[\H^2(G,M)\to\H^2(N,M)^{\overbar G}]
\lr \H^1(\overbar G,\H^1(N,M))\lr \H^3(\overbar G,M^{N})
\end{align*}

\section{Standard Setup}

\subsection{}\label{setup}
Let $F=(F,\v)$ be a henselian-valued field with finite residue field $k$ 
of cardinality $|k|=q=p^f$, and totally ordered value group $\Gamma_F\isom\Z^d$ of rank $d\geq 1$.
In the following,
\begin{enumerate}[\quad $\circ$]
\item
$\ell\neq p$ is a prime
\item
$n$ is a power of $\ell$
\item
$m=|\mu_n(k)|=|\mu_n(F)|=\gcd(q-1,n)$
\item
$m'=m/2$ if $m=2^{\v_2(q-1)}$, and $m'=m$ otherwise.
\end{enumerate}
Note $m=1$ if and only if $\mu_\ell\not\leq F^\times$, 
and $m=n$ if and only if $\mu_n\leq F^\times$.

Let $F_\nr$ denote the maximal unramified extension of $F$, which is the strict henselization of $F$,
and let $F_\ttr$ be the maximal tamely ramified extension of $F$, which is obtained from $F_\nr$ by 
adjoining all prime-to-$p$-th roots of elements of $F$, by \cite[Proposition A.22]{TW15}.
We identify $\G_k$ with $\Gal(F_\nr/F)$, and put $\G_F^\ttr=\Gal(F_\ttr/F)$.

\subsection{Index in the Tame Brauer Group}\label{value}
Suppose $F$ is henselian with finite residue field $k$, and $D/F$ is a division algebra.
To compute index we will use the valuation theoretic framework of Tignol-Wadsworth's \cite{TW87}
and Jacob-Wadsworth's \cite{JW90},
which we summarize. 

Since $k$ is finite, the residue division algebra $\overbar D$ is a field extension of $k$,
by Wedderburn's Theorem.
The surjective homomorphism
$\theta_D:\Gamma_D/\Gamma_F\to\Gal(Z(\overbar D)/k)$ of \cite[p.96]{DK80} (see \cite[Proposition 1.7]{JW90})
has kernel $\Gamma_T/\Gamma_F$ by \cite[Theorem 6.3]{JW90}, and it follows that 
$[\overbar D:k]=[\Gamma_D:\Gamma_T]$.
Therefore by Draxl's Ostrowski Theorem \cite[Theorem 2]{Dr84},
if the index of $D$ is prime-to-$\car(k)$, then
\[[D:F]=[\Gamma_D:\Gamma_T][\Gamma_D:\Gamma_F]\]
The index of $D$ is the square root of $[D:F]$.
By \cite[Decomposition Lemma 6.2]{JW90} (see also \cite[Proposition 8.59]{TW15}), $D\sim S\otimes_F T$
where $S$ and $T$ are nicely semiramified and totally ramified $F$-division algebras,
and $\Gamma_D=\Gamma_S+\Gamma_T$ by \cite[Theorem 6.3]{JW90}.

\section{Multiplicative Group}\label{unifsbg}
Assume \eqref{setup}.
With $\overbar G=\G_k$, $G=\G_F$, $N=\G_{F_\nr}$, and $M=\mu_n$, 
\eqref{7term} yields
\[1\lr k^\times/n\lr F^\times\!/n\lr F_\nr^\times/n\lr 1\] 
and composing with the valuation map $\v:F_\nr^\times/n\isim\Gamma_{F_\nr}/n=\Gamma_F/n$, 
we obtain the usual valuation sequence on $F^\times\!/n$. This sequence splits with the choice of a 
{\it uniformizer subgroup} $T_n$ of $F^\times\!/n$,
which is the group generated by a basis of {\it uniformizers} $\{(x_1)_n,\dots,(x_d)_n\}$: 
elements of $F^\times\!/n$ whose values $\{\v(x_1),\dots,\v(x_d)\}$ form a basis of $\Gamma_F/n$.
Since $k^\times$ is a finite cyclic group and $\mu_n(k)=\mu_m(k)\isom\mu_m(F)$, 
we have $k^\times/n=k^\times/m\isom\Z/m$. 
Thus 
\begin{equation}\label{multgp}
F^\times\!/n\isom k^\times/m\times T_n\isom\mu_m(F)\times T_n\isom\Z/m\oplus(\Z/n)^d
\end{equation}
and the natural map $F^\times\!/n\to F_\nr^\times/n$ maps $T_n$ isomorphically onto $F_\nr^\times/n$.
Similarly, the sequence
\[1\lr k^{\times 1/n}/k^\times\lr F^{\times1/n}/F^\times\lr F_\nr^{\times1/n}/F_\nr^\times\lr 1\] 
shows
\[F^{\times 1/n}/F^\times\isom\mu_{mn}/\mu_m\times\langle t_1^{1/n},\dots,t_d^{1/n}\rangle\isom(\Z/n)^{d+1}\]
for a uniformizer basis $\{t_1^{1/n},\dots,t_d^{1/n}\}$ for $F^{\times 1/n}/F^\times$.
For $|k^{\times 1/n}|=|\zeta_{n(q-1)}|=n(q-1)$, so $k^{\times 1/n}/k^\times$ has order $n$. 
If $t_0^{1/n}$ generates $\mu_{mn}/\mu_m$, 
the set $\{t_0^{1/n},\dots,t_d^{1/n}\}$ forms a basis for $F^{\times 1/n}/F^\times$.
We will say the basis is in {\it standard form} if its last $d$ elements make a uniformizer basis.

\section{Galois Group and Character Group}

The structure theory of the Galois group of the maximal tamely 
ramified extension of a local field goes back at least to Iwasawa (\cite{I55});
see \cite[VII.5]{NSW} for additional background.
Assume \eqref{setup}. 
We have a split exact sequence
\[\begin{tikzcd}
1\arrow[r]&\G_{F_\nr}^\ttr\arrow[r]
&\G_F^\ttr\arrow[r]&\G_k\arrow[r]\arrow[l, dashed, shift left=5pt, "s"]&1 
\end{tikzcd}\]
whereby $\G_F^\ttr=\G_k\ltimes\G_{F_\nr}^\ttr$, 
$\G_{F_\nr}^\ttr\isom(\wZ')^d$, and $\G_k\isom\wZ$.
Let $\Phi_0= s(1)$ be a Frobenius generator, which exponentiates by $q$.
Let $\{x_1,\dots,x_d\}$ be a uniformizer basis for $F$, generating a subgroup $T\leq F^\times$.
Let $\{\Phi_1,\dots,\Phi_d\}$ be a (topological) basis for $\G_{F_\nr}^\ttr$ dual to $\{x_1,\dots,x_d\}$,
so that for each $n$, $x_j^{1/n}(\Phi_i)=\zeta_n^{\delta_{ij}}$, $i,j\geq 1$.
Let $\un\Phi=\{\Phi_0,\dots,\Phi_d\}$ be the resulting basis for $\G_F^\ttr$.
Then we have the presentation
\begin{equation}\label{origpres}
\G_F^\ttr=\br{\Phi_0,\dots,\Phi_d:\,[\Phi_0,\Phi_j]=\Phi_j^{q-1},\,\forall j\geq 1;\,
[\Phi_i,\Phi_j]=e,\,\forall i,j\geq 1}
\end{equation}

\begin{Theorem}\label{chargp}
Assume \eqref{setup}, let $T_{q-1}\leq F^\times/(q-1)$ be a uniformizer group, 
and let $T_{q-1}^*=\zeta_{q-1}^*(T_{q-1})$ as in \eqref{zetam}.
Then $\X(F^\ttr/F)\isom\X(k)\times T_{q-1}^*\isom\Q/\Z\oplus(\tfrac 1{q-1}\Z/\Z)^{\oplus d}$,
and \[{}_n\X(F)\isom\tfrac 1n\Z/\Z\oplus(\tfrac 1m\Z/\Z)^{\oplus d}\]
\end{Theorem}

\begin{proof}
From \eqref{origpres} we compute 
$[\G_F^\ttr,\G_F^\ttr]=(\G_{F_\nr}^\ttr)^{q-1}\isom(q-1)\widehat\Z^{\prime d}$,
so 
\[\X(F^\ttr/F)=\Hom(\G_F^\ttr,\Q/\Z)=\Hom(\widehat\Z\times\tfrac{\widehat\Z^{\prime d}}{q-1},\Q/\Z)
\isom\Q/\Z\oplus(\tfrac 1{q-1}\Z/\Z)^{\oplus d}\]
The first factor is $\X(k)$, 
and since $\G_{F_\nr}^\ttr/(\G_{F_\nr}^\ttr)^{q-1}\isom\Gal(F_\nr^{1/(q-1)}/F_\nr)$, 
the second factor is $T_{q-1}^*$ by Kummer theory, for any uniformizer subgroup $T_{q-1}$.
The last statement is immediate.
\end{proof}

\section{Brauer Group}

\subsection{Brauer Group of $F_\nr$}\label{strictBrgp}
Let $F=(F,\v)$ be henselian-valued of rank $d$, with residue field $k$ either finite, as in \eqref{setup}, 
or algebraically closed, in which case $F$ is strictly henselian.
The tame Brauer group of a strictly henselian field is naturally isomorphic to
a group of alternating forms on
the absolute Galois group, by \cite{Br01b} and \cite{Br07}, 
giving a method for finding a minimal expression and an index formula for each class.
We prove a slight variant.

Assume $\mu_n\leq F^\times$.
Let $G=\Gal(L/F)$, where either $L=F^{1/n}$, or $L=\varinjlim F^{1/n}$,
with the limit over all prime-to-$p$ numbers $n$ if $k$ is algebraically closed.
Then $G$ is isomorphic to either $(\Z/n)^{d+1}$, $(\Z/n)^d$, or $\widehat\Z^{\prime d}$.
Let $\Alt(G,\mu_n)$ denote the set of continuous alternating bilinear functions on $G$.

\begin{Lemma}\label{uct}
In the setup above, with $\mu_n\leq F^\times$,
there is a commutative diagram
\[\begin{tikzcd}
0\arrow[r]&K\arrow[d]\arrow[r]
&\H^1(F,\tfrac 1n\Z/\Z)\otimes\H^1(F,\mu_n)\arrow[d, "\cup"]\arrow[r,"\alt"]
&\Alt(G,\mu_n)\arrow[d,equals]\arrow[l, dashed, shift left=5pt]\arrow[r]&0\\
0\arrow[r]&\Ext_\Z^1(G,\mu_n)\arrow[r]&\H^2(G,\mu_n)\arrow[r,"{[\alt]}"]&\Alt(G,\mu_n)\arrow[r]
\arrow[l, dashed, shift left=5pt]&0
\end{tikzcd}\]
where $\alt([f])(\sigma,\tau):=f(\sigma,\tau)/f(\tau,\sigma)$ for $f\in\cZ^2(G,\mu_n)$.
If $G$ is torsion-free, then $[\alt]$ is an isomorphism,
and $\H^2(G,\mu_n)$ is generated by cyclic classes.
If $G=(\Z/n)^r$, then $\Ext_\Z^1(G,\mu_n)\isom\mu_n^r$.
\end{Lemma}

\begin{proof}
The result appears in \cite[Lemma 2.2, Theorem 2.4]{Br07} with $\Q/\Z$ in place of $\tfrac 1n\Z/\Z$ and $\mu_n$.
It also follows from the Universal Coefficient Theorem.

We may assume $n$ is a power of a prime $\ell$ by primary decomposition.
The group $G$ is either isomorphic to $(\Z/n)^r$, where $r=d+1$ or $d$, 
or $(\widehat\Z')^d$ if $k$ is algebraically closed.

Viewing $\H^1(F,\tfrac 1n\Z/\Z)\otimes\H^1(F,\mu_n)$ as a subgroup of $\cZ^2(G,\mu_n)$,
extend $\alt$ to all of $\text{\rm Z}^2(G,\mu_n)$ by setting
$\alt(f)=f/\tilde f$ for all $f\in\cZ^2(G,\mu_n)$,
where $\tilde f(\sigma,\tau):=f(\tau,\sigma)$.
It is straightforward to verify that $\tilde f$ is a $2$-cocycle,
because $G$ is abelian and acts trivially on $\mu_n$.
Moreover, $f/\tilde f$ is indeed an alternating form on $G$.
For it is clear from the definition that 
$(f/\tilde f)(\sigma,\tau)=(f/\tilde f)(\tau,\sigma)^{-1}$,
and the alternating sum of the cocycle condition on $f$ applied successively
to the triples $(\rho,\sigma,\tau),(\rho,\tau,\sigma)$, and $(\tau,\rho,\sigma)$
shows that $f/\tilde f$ is linear on the left, hence bilinear
(see \cite[Section 1]{AT85}).

If $f\in\cC^1(G,\mu_n)$ then since $G$ is abelian and acts trivially on $\mu_n$,
$\partial f\in\cB^2(G,\mu_n)$ is in 
the subgroup $\cZ^2(G,\mu_n)_\sym$ of $\cZ^2(G,\mu_n)$ generated by elements $f:f=\tilde f$.
Therefore $[\alt]$ is well-defined on $\H^2(G,\mu_n)$, and we have a commutative ladder
\[\begin{tikzcd}
0\arrow[r]&K\arrow[d]\arrow[r]
&\H^1(F,\tfrac 1n\Z/\Z)\otimes\H^1(F,\mu_n)\arrow[d, "\cup"]\arrow[r,"\alt"]
&\Alt(G,\mu_n)\arrow[d,equals]\arrow[r]&0\\
0\arrow[r]&\H^2(G,\mu_n)_\sym\arrow[r]&\H^2(G,\mu_n)\arrow[r,"{[\alt]}"]&\Alt(G,\mu_n)\arrow[r]
&0
\end{tikzcd}\]
If $\{(x_i)_n\}$ is a basis for $\H^1(F,\mu_n)$, the
elements $\alt((x_i)_n^*\otimes(x_j)_n)$ 
for $i<j$ are easily seen to form a basis for $\Alt(G,\mu_n)$, hence $\alt$ splits.
Therefore $[\alt]$ splits.
The group $\H^2(G,\mu_n)_\sym$ describes the abelian group-extensions of $G$ by $\mu_n$,
hence it is isomorphic to $\Ext_\Z^1(G,\mu_n)$. 
If $G=(\Z/n)^r$, then $\Ext_\Z^1(G,\mu_n)=\Hom_\Z(\Z^r,\mu_n)\isom\mu_n^r$ is a standard
computation, see e.g. \cite[Ch.17]{DF}.
If $G$ is torsion-free then it is a direct limit of projective $\Z$-modules,
hence $\Ext_\Z^1(G,\mu_n)=0$,
and $[\alt]$ is an isomorphism.
Then the commutative ladder and the Snake Lemma show
$\H^2(G,\mu_n)$ is generated by elements $\theta\ccup(t)_n$,
for $\theta\in\H^1(F,\tfrac 1n\Z/\Z)$ and $(t)_n\in\H^1(F,\mu_n)$.
\end{proof}

\begin{Remark}
We study the connection between $\H^2(G,\mu_n)$ and $\H^2(F,\mu_n)$ in Section~\ref{naturalmaps}.
\end{Remark}

\subsection{Brauer Group of $F$}
\begin{Theorem}\label{Brgp}
Assume \eqref{setup}.
Then there is an exact sequence
\[0\lr{}_n\H^2(\G_k,F_\nr^\times)\lr\H^2(F,\mu_n)\lr\H^2(F_\nr,\mu_m)\lr 0\]
split by a choice of uniformizer group,
so that \[{}_n\Br(F)\isom(\tfrac 1n\Z/\Z)^{\oplus d}\oplus(\tfrac 1m\Z/\Z)^{\oplus d(d-1)/2}\]
Suppose $\chi_0\in {}_n\X(F)$ is unramified of order $n$, 
and $\{(x_1)_n,\dots,(x_d)_n\}$ is a uniformizer basis.
Then each $\alpha\in{}_n\Br(F)$ has a unique coordinate expression
\[\alpha\;=\;\sum_{j=1}^d n_{0j}(\chi_0,x_j)+\sum_{1\leq i<j\leq d}m_{ij}(x_i,x_j)_m\]
\end{Theorem}

\begin{proof}
This is well-known, see e.g. \cite[Theorem 7.84]{TW15}.
Since $F_\nr/F$ is Galois, $\G_{F_\nr}$ is a closed normal subgroup of $\G_F$.
Let $\G_k=\Gal(F_\nr/F)$.
Since $\H^1(\G_{F_\nr},F_\sep^\times)=0$ by Hilbert 90,
the $n$-torsion of the 7-term sequence \eqref{7term} applied to $M=F_\sep^\times$
yields the inflation-restriction sequence
\[0\lr{}_n\H^2(\G_k,F_\nr^\times)\overset{\inf}{\lr}\H^2(F,\mu_n)
\overset{\res}{\lr}\H^2(F_\nr,\mu_n)^{\G_k}\]
By Lemma~\ref{uct}, $\H^2(F_\nr,\mu_n)$ is generated by
cyclic classes $(x)_n^*\ccup(y)_n$. 
If $\{(x_1)_n,\dots,(x_d)_n\}$ is a basis for a uniformizer subgroup
$T_n$ for $F^\times\!/n$, then since $T_n\isom F_\nr^\times/n$, 
every element of $\H^2(F_\nr,\mu_n)$ is uniquely expressible
in terms of the $(x_i)_n^*\ccup(x_j)_n=(x_i,x_j)_n$ with $1\leq i<j\leq d$.
Thus $\H^2(F_\nr,\mu_n)\isom\mu_n^{d(d-1)/2}$ by counting.
The explicit action of $\G_k$ shows this is a $\G_k$-module isomorphism.
Since $\mu_n^{\G_k}=\mu_m$,
it follows that 
$\H^2(F_\nr,\mu_n)^{\G_k} 
\isom\mu_m^{d(d-1)/2}$, and
we have an exact sequence
\[0\lr{}_n\H^2(\G_k,F_\nr^\times)\lr\H^2(F,\mu_n)\lr\H^2(F_\nr,\mu_m)\lr 0\]
with a splitting defined by the choice of $T_n\leq F^\times\!/n$.

Since $\H^q(\G_k,\mu_n)=0$ for $q\geq 2$,
the 7-term sequence \eqref{7term} with $M=\mu_n$ defines a natural isomorphism 
\[{}_n\H^2(\G_k,F_\nr^\times)\isim\H^1(\G_k,\H^1(\G_{F_\nr},\mu_n))
\]
An easy computation shows that the action of $\G_k$ on $\H^1(\G_{F_\nr},\mu_n)$ is trivial.
Therefore since $\H^1(\G_{F_\nr},\mu_n)\isom F_\nr^\times/n\isom\Gamma_F/n$, we find
${}_n\H^2(\G_k,F_\nr^\times)\isom\H^1(\G_k,\Gamma_F/n)\isom\br{\chi_0}^d$, 
where $\chi_0\in\X(k)$ is any character of order $n$.
A tracing through the maps shows the splitting $\br{\chi_0}^d\to{}_n\H^2(\G_k,F_\nr^\times)$ is induced by the choice
of basis $\{(x_1)_n,\dots,(x_d)_n\}$ for $T_n$, and sends the $j$-th copy of $\chi_0$
to $(\chi_0,x_j)$. Thus a general class is uniquely expressible as a sum
$\sum_{j=1}^d n_{0j}(\chi_0,x_j)$.
We conclude
$\H^2(F,\mu_n)\isom\mu_n^d\times\mu_m^{d(d-1)/2}$, and 
each $\alpha\in\H^2(F,\mu_n)$ is uniquely expressible in the form
\[\begin{matrix*}[l]
&n_{01}(\chi_0,x_1) &+ &n_{02}(\chi_0,x_2)&+&\cdots&+&n_{0d}(\chi_0,x_d)\\
& &+&m_{12}(x_1,x_2)_m&+&\cdots&+&m_{1d}(x_1,x_d)_m\\
& & & & & & &\qquad\vdots\\
& & & & & &+&m_{d-1d}(x_{d-1},x_d)_m
\end{matrix*}\]
\end{proof}

\section{Brauer Dimension and Cyclic Length}\label{Brdim}

\begin{Definition}
Let $F=(F,\v,k)$ be a valued field with residue field $k$.
\begin{enumerate}[(a)]
\item
The {\it tame cyclic length} 
of $F$ is the the minimum number of cyclic classes of degree $n$ 
needed to express all elements of ${}_n\Br(F)$, over all $n$ invertible in $k$.
It is zero if $\Br(F)=0$, and $\infty$ if no such number exists.
\item
The {\it tame Brauer dimension} of $F$ is the supremum 
of the smallest $d\geq 0$ such that 
$\ind(\alpha)$ divides $\per(\alpha)^d$ for all tame $\alpha\in\Br(F)$, or $\infty$ if no such number exists.
The relation $\ind(\alpha)\,|\,\per(\alpha)^d$ for all $\alpha$ is called a {\it period-index bound}.
\item
The {\it maximum period-index ratio in ${}_n\Br(F)$} is the supremum of $\ind(\alpha)/\per(\alpha)$
over all $\alpha\in{}_n\Br(F)$. It is zero if ${}_n\Br(F)=0$.
\end{enumerate}
\end{Definition}

\begin{Remark}
Since a class of period $n$ and $n$-cyclic length $d$ has index at most $n^d$,
Brauer dimension is bounded by cyclic length.
Cyclic length and Brauer dimension are sensitive to the presence of roots of unity.
For example, all tame division algebras over $\F_2((t_1))((t_2))((t_3))$ are cyclic of equal period and index, 
whereas, as the following example shows, $\F_4((t_1))((t_2))((t_3))$ has elements of cyclic length $2$
and unequal period and index.
\end{Remark}

The following example gives a lower bound, which turns out to be an upper bound.

\begin{Example}\label{max}
Assume \eqref{setup}, with $n$ a prime-power and $m\neq 1$. 
Let $\chi_0$ be an unramified character of order $n$, and $\{(x_1)_n,\dots,(x_d)_n\}$ a uniformizer basis. 
The algebra
\[D\;=\;\Delta(\chi_0,x_1)\otimes_F \Delta(x_2,x_3)_m\otimes_F \Delta(x_4,x_5)_m\otimes_F\cdots\otimes_F
\Delta(x_{2r},x_{2r+1})_m\]
where $r=\lfloor\frac{d-1}2\rfloor$,
is a division algebra of period $n$ and index $nm^r$, and cyclic length $r+1=\lfloor\frac{d+1}2\rfloor$,
by \cite[Corollary 2.6]{JW86} and \cite[Extension Lemma 1.6]{TW87}.
This example is well known.
\end{Example}

\begin{Theorem}
Assume \eqref{setup}.
\begin{enumerate}[\rm (a)]
\item
If $k=\F_2$, the tame Brauer dimension and cyclic length are both $1$,
and the maximum period-index ratio for ${}_n\Br(F)$ is $1$, for all (prime-to-$p$) $n$.
\item
If $k\neq\F_2$, the tame Brauer dimension and cyclic length are both $\lfloor\tfrac{d+1}2\rfloor$,
and the maximum period-index ratio in ${}_n\Br(F)$ is $m^{\lfloor\tfrac{d-1}2\rfloor}$, where 
$m=|\mu_n(k)|$, for all $n$.
\end{enumerate}
\end{Theorem}

\begin{proof}
The case $k=\F_2$ is immediate, since there are no tame totally ramified classes, so
${}_n\Br(F)=\{(\chi_0,t):t\in F^\times/\Nm(F(\chi_0)^\times)\}$ for an unramified
character $\chi_0$ of order $n$. 

Now suppose $k\neq\F_2$, so that $F$ has a nontrivial tame root of unity, and $m\neq 1$ for some 
prime-to-$p$ number $n$.
We may assume without loss of generality that $n$ is a prime power.

Cyclic length. 
Suppose $D\isom D_0\otimes_F D_1\otimes_F\cdots\otimes_F D_r$ is a decomposition into cyclic
$F$-division algebras of period dividing $n$.
Since $k$ is finite, $\overbar D/k$ is a finite cyclic field extension,
and since disjoint division algebras cannot have common subfields,
at most one of the $D_i$, say $D_0$, has a (nontrivial) unramified subfield, which maps onto $\overbar D$.
If $D$ has no unramified subfields, we set $D_0=F$.
Let $T=D_1\otimes_F\cdots\otimes_F D_r$.
Since $T$ is tame and has no unramified subfields, it is totally ramified. 
The valuation $\v$ extends uniquely to each $D_i$, since $F$ is henselian,
and for $i\geq 1$ the $\Gamma_{D_i}/\Gamma_F$
are all rank $2$ and mutually disjoint, so the rank of $\Gamma_T/\Gamma_F$ is $2r$.
Since $\Gamma_D\subset\tfrac 1n\Gamma_F$, the rank of $\Gamma_D/\Gamma_F$ is at most $d$,
so $2r\leq d$.
Since $D_0$ is not inertial, $\Gamma_{D_0}/\Gamma_F$ has rank at least $1$, 
so $D_0$ can only contribute to a maximum cyclic length if $d$ is odd, and then we compute
the maximum is $1+(d-1)/2=(d+1)/2$. If $d$ is even, then $D_0$ isn't necessary to achieve
the maximum, which is $r=d/2$. 
Thus in any case $\lfloor\tfrac{d+1}2\rfloor$ is an upper bound to the cyclic length,
and this is realized by Example~\ref{max}. 

Brauer dimension.
Let $D$ be an $F$-division algebra of period $n$, and as in Section~\ref{value}, let 
$S$ and $T$ be semiramified and totally ramified division algebras such that $D\sim S\otimes_F T$,
$\Gamma_D=\Gamma_S+\Gamma_T$, and $[D:F]=[\Gamma_D:\Gamma_T]^2[\Gamma_T:\Gamma_F]$.
Since $k$ is finite and 
$S$ is semiramified, $S$ is cyclic of degree dividing $n$, and $\Gamma_D/\Gamma_T$ is cyclic of order dividing $n$.
The group $\Gamma_T/\Gamma_F$ has even rank at most $d$, the rank of $\tfrac 1m\Gamma_F/\Gamma_F$.
If the rank equals $d$, and the minimum exponent of a summand is $m_0$, then $\Gamma_D/\Gamma_T$
has order dividing $n/m_0$, hence $[D:F]$ divides $n^2 m^{d-2}$, with $d$ even.
If the rank is less than $d$, then we could have $\Gamma_D/\Gamma_T=n$, and $\Gamma_T/\Gamma_F$ has
order dividing $m^{d-1}$ if $d$ is odd, and order dividing $m^{d-2}$ if $d$ is even.
Therefore $[D:F]$ divides $n^2 m^{d-1}$ if $d$ is odd, $n^2 m^{d-2}$ if $d$ is even.
In any case, we find
$\ind(D)$ divides $nm^{\lfloor\tfrac{d-1}2\rfloor}$,
which is the lower bound of Example~\ref{max}. We compute from this that the 
(tame) Brauer dimension is $\lfloor\tfrac{d+1}2\rfloor$, and the
maximum period-index ratio is thus $m^{\lfloor\tfrac{d-1}2\rfloor}$, as claimed.
Note the latter remains true if $m=1$.
\end{proof}

\section{Wedge Product}

Since ${}_n\Br(F)$ is generated by cyclic classes, the cup product map
\[{}_n \H^1(F,\tfrac 1n\Z/\Z)\otimes\cZ^1(F,\mu_n)\lr{}_n \Br(F)\]
is surjective.
In this section we show the extent to which the cokernel is a wedge product. 

\begin{Example}\label{fb}
Assume \eqref{setup}, with $n$ a power of $\ell$.
Let $\chi_0$ be the Frobenius character of order $n$,
set $x_0^{1/n}=\zeta_{n(q-1)}^{q}$, and let $\{x_1^{1/n},\dots,x_d^{1/n}\}$
be a uniformizer basis for $F^{\times 1/n}/F^\times$.
If we put $\chi_i=(x_i)_m^*$ for $i\geq 1$ then 
$\un\chi=\{\chi_0,\dots,\chi_d\}$ is a basis for ${}_n\X(F)$ by Theorem~\ref{chargp}, and
$\un x=\{x_0^{1/n},\dots,x_d^{1/n}\}$ is a basis for $\cZ^1(F,\mu_n)$ by 
Section \eqref{unifsbg}.
In the following we will use two properties of this basis:
\begin{enumerate}[\rm (a)]
\item\label{fba}
$(\chi_0,x_0)=0$ 
\item\label{fbb}
$\tfrac nm\chi_0=(x_0)_m^*$ and $\chi_i=(x_i)_m^*$
\end{enumerate}
\eqref{fba} is by Wedderburn's Theorem, since $(\chi_0,x_0)$ is
defined over the finite field $k$.
The first part of \eqref{fbb} is because $\tfrac nm\chi_0$ and $(x_0)_m^*$ are defined over $k$,
and agree on the Frobenius automorphism, the second part is by definition.

\end{Example}

\begin{Theorem}\label{tt}
Assume \eqref{setup}, with $n$ a power of $\ell$.
Let $\un\chi\times\un x$ be a basis satisfying \eqref{fba},\eqref{fbb} of Example~\ref{fb}, 
put
\[
\theta=\sum_{i=0}^d b_i\chi_i\in{}_n\X(F)\quad\text{and}\quad
t^{1/n}=\prod_{i=0}^d x_i^{c_i/n}\in\cZ^1(F,\mu_n) 
\]
and assume $|t^{1/n}|=n$. Consider the equations
\begin{equation}\label{eqns}
b_0 c_j=\tfrac nm c_0 b_j\pmod n\quad\text{and}\quad b_i c_j\equiv c_i b_j\pmod m\quad\forall\;i,j\geq 1
\end{equation}
\begin{enumerate}[\rm(a)]
\item\label{tta}
If $\theta\in{}_n2\,\X(F)$ or $|\theta|\neq m$, then \eqref{eqns} is equivalent to $(\theta,t)=0$.
\item\label{ttb}
If \eqref{eqns} holds, then $(\theta,t)= 0$ is equivalent to
$\theta\in{}_n2\,\X(F)$ or $|\theta|\neq m$.
\item\label{ttc}
Let $g=\gcd(|\theta|,m)$. 
Then \eqref{eqns} implies $\langle\tfrac{|\theta|}g\theta\rangle=\br{(t)_g^*}$. 
\end{enumerate}
\end{Theorem}

\begin{proof}
Compute using properties \eqref{fba},\eqref{fbb} of Example~\ref{fb}, 
\begin{align*}
(\theta,t)&=(\sum_{i=0} b_i\chi_i,\,\prod_{i=0} x_i^{c_i})
=(b_0\chi_0,\,\prod_{i=0} x_i^{c_i})+(\sum_{i=1} b_i\chi_i,\,x_0^{c_0})+
(\sum_{i=1} b_i\chi_i,\,\prod_{i=1} x_i^{c_i})\\
&=\sum_{i=1}(b_0 c_i-\tfrac nm c_0 b_i)(\chi_0,x_i)+\sum_{1\leq i<j}(b_i c_j-b_j c_i)(x_i,x_j)_m
+\sum_{i=1}b_i c_i(x_i,x_i)_m
\end{align*}
If $\theta\in{}_n2\,\X(F)$ then $b_i c_i(x_i,x_i)_m=0$ for each $i\geq 1$
by \eqref{cyclic} and Albert's criterion, 
so $(\theta,t)=0$ if and only if \eqref{eqns} holds.

Suppose $\theta\not\in{}_n2\,\X(F)$, so $\ell=2$, $m=2^{\v_2(q-1)}$ 
divides $|\theta|$, and $2\nmid b_i$ for some $i\geq 1$.
If $|\theta|\neq m$ then $|\theta|=|b_0\chi_0|>m$,
so $n/m\nmid b_0$, and \eqref{eqns} implies $2|\, c_i$ for each $i\geq 1$, hence $(\theta,t)=0$.
Conversely, if $|\theta|\neq m$ and $(\theta,t)=0$,
then $m(\theta,t)=\sum_{i=1}mb_0 c_i(\chi_0,x_i)=0$,
and since $n/m\nmid b_0$, and $(\chi_0,x_i)$ has order $n$,
we must have $2|\, c_i$ for each $i\geq 1$. Therefore $\sum_{i=1}b_i c_i(x_i,x_i)_m=0$,
and \eqref{eqns} holds.
This proves \eqref{tta}.

Assume \eqref{eqns}. If $\theta\in{}_n2\,\X(F)$ or $|\theta|\neq m$ then $(\theta,t)=0$ by \eqref{tta}.
Conversely, if $(\theta,t)=0$, then because of \eqref{eqns},
$b_ic_i(x_i)_m^*\in{}_n2\,\X(F)$ for each $i\geq 1$ by Albert's criterion, and
\begin{equation}\label{always}
c_i\theta-b_i(t)_m^*=(b_0 c_i-\tfrac nm c_0 b_i)\chi_0+\sum_{j=1}(b_j c_i-b_ic_j)(x_j)_m^*=0\qquad (i\geq 1)
\end{equation}
If $(\theta,t)=0$ and $|\theta|=m$, then since $t^{1/n}$ has order $n$, $|\theta|=|(t)_m^*|=m$
hence $b_i, c_i$ are units by \eqref{always}. Since $b_i c_i(x_i)_m^*\in{}_n2\,\X(F)$, 
we must have $m\neq 2^{\v_2(q-1)}$,
hence $\theta\in{}_n2\,\X(F)$. This proves \eqref{ttb}.

To prove \eqref{ttc}, suppose \eqref{eqns}, then
\[c_0\tfrac nm\theta-b_0(t)_m^*=(c_0 b_0-b_0 c_0)\tfrac nm\chi_0
+\sum_{j=1}(\tfrac nm c_0 b_j-b_0 c_j)(x_j)_m^*=0\]
If $\ell\nmid b_i$ for some $i\geq 0$,
then $\theta$ is not extendable in ${}_n\X(F)$, and $b_i(t)_m^*$ has order $m$, 
hence $\langle\frac{|\theta|}m\theta\rangle=\br{(t)_m^*}$ by \eqref{always}.
If $\theta$ extends to a non-extendable $\theta'$ in ${}_n\X(F)$,
then $\tfrac mg\tfrac{|\theta'|}m\theta'=\tfrac{|\theta|}g\theta$,
hence $\langle\tfrac{|\theta|}g\theta\rangle=\langle(t)_g^*\rangle$.
Therefore \eqref{eqns} implies $\langle\tfrac{|\theta|}g\theta\rangle=\langle(t)_g^*\rangle$ in any case.
\end{proof}

\begin{Remark}\label{linind}
The relations \eqref{eqns} are that the matrix
\[\bmx{\theta\\t^{1/n}}=\bmx{b_0&\tfrac nmb_1&\cdots&\tfrac nm b_d\\c_0&c_1&\cdots&c_d}\;\pmod n\]
have rank one, a linear dependence condition.
By Theorem~\ref{tt}\eqref{tta}, this ``linear dependence'' 
is not by itself enough to imply $(\theta,t)=0$.
For example, if $m=2^{\v_2(q-1)}$, $t^{1/n}$ is a uniformizer, and $\theta=(t)_m^*$,
then \eqref{eqns} holds, but $(t,t)_m\neq 0$.
Therefore
to achieve our goal of casting ${}_n\Br(F)$ as a wedge product, we will restrict to the subgroup
${}_n2\,\Br(F)\leq {}_n\Br(F)$
defined by characters in ${}_n2\,\X(F)$. To that end we have the following corollary.
\end{Remark}

\begin{Corollary}\label{tt'}
Assume \eqref{setup}, and $\un\chi\times\un x$ is a basis for ${}_n2\,\X(F)\times\cZ^1(F,\mu_n)$ 
satisfying 
\begin{enumerate}[\rm(a$'$)]
\item\label{fba'}
$(\chi_0,x_0)=0$
\item\label{fbb'}
$\tfrac{|\chi_i|}{g_i'}\chi_i=(x_i)_{g_i'}^*$, where $g_i'=\gcd(|\chi_i|,m')$, for $i\geq 0$.
\end{enumerate}
Suppose
$\ds{
\theta=\sum_{i=0}^d b_i\chi_i\in{}_n2\,\X(F)}$ and $\ds{t^{1/n}=\prod_{i=0}^d x_i^{c_i/n}\in\cZ^1(F,\mu_n) 
}$,
with $|t^{1/n}|=n$. Then $(\theta,t)=0$ if and only if
\begin{equation}\label{eqns'}
b_0 c_j=\tfrac n{m'} c_0 b_j\pmod n\quad\text{and}\quad b_i c_j\equiv c_i b_j\pmod{m'}\qquad
\forall\,i,j\geq 1
\end{equation}
If $g'=\gcd(m',|\theta|)$, then \eqref{eqns'} implies $\langle\frac{|\theta|}{g'}\theta\rangle=\langle(t)_{g'}^*\rangle$. 
\end{Corollary}

\begin{proof}
The proof that $(\theta,t)=0$ is equivalent to \eqref{eqns'} is exactly as in
paragraph one of Theorem~\ref{tt}'s proof, with $m'$ in place of $m$,
using (\ref{fba'}$'$) and (\ref{fbb'}$'$), and noting that 
$\theta\in{}_n2\,\X(F)$ by hypothesis.
Similarly, if \eqref{eqns'} holds then imitating the proof of Theorem~\ref{tt}, we have
\begin{align*}
c_0\tfrac n{m'}\theta-b_0(t)_{m'}^*&=(c_0 b_0-b_0 c_0)\tfrac n{m'}\chi_0
+\sum_{j=1}(\tfrac n{m'} c_0 b_j-b_0 c_j)(x_j)_{m'}^*=0\\
c_i\theta-b_i(t)_{m'}^*&=(b_0 c_i-\tfrac n{m'} c_0 b_i)\chi_0+\sum_{j=1}(b_j c_i-b_ic_j)(x_j)_{m'}^*=0\qquad (i\geq 1)
\end{align*}
If $\ell\nmid b_i$ for some $i\geq 0$,
then $\theta$ is not extendable in ${}_n2\,\X(F)$, and $b_i(t)_{m'}^*$ has order $m'$, 
hence $\langle\frac{|\theta|}{m'}\theta\rangle=\br{(t)_{m'}^*}$.
If $\theta$ extends to a non-extendable $\theta'$ in ${}_n2\,\X(F)$,
then $\tfrac {m'}{g'}\tfrac{|\theta'|}{m'}\theta'=\tfrac{|\theta|}{g'}\theta$,
hence $\langle\tfrac{|\theta|}{g'}\theta\rangle=\langle(t)_{g'}^*\rangle$.
Therefore \eqref{eqns'} implies $\langle\tfrac{|\theta|}{g'}\theta\rangle=\langle(t)_{g'}^*\rangle$ in any case.
\end{proof}

\begin{Definition}\label{wedge}
Assume \eqref{setup}, with $n$ a power of $\ell$.
\begin{enumerate}[\rm(a)]
\item\label{wedgea}
$\theta\in{}_n2\,\X(F)$ and $t^{1/n}\in\cZ^1(F,\mu_n)$ are {\it matched}
if they satisfy \eqref{eqns'} with respect to a basis satisfying
(\ref{fba'}$'$) and (\ref{fbb'}$'$) of Corollary~\ref{tt'}, and we have the normalization
$\tfrac{|\theta|}{g'}\theta=(t)_{g'}^*$, where $g'=\gcd(|\theta|,m')$,
which in the above notation is equivalent to 
\[b_0\equiv\tfrac n{|\theta|}c_0\pmod{\tfrac{ng'}{|\theta|}}
\quad\text{and}\quad\tfrac{|\theta|}{g'}b_i\equiv\tfrac{m'}{g'}c_i\pmod{m'}\quad(i\geq 1)\]
\item\label{wedgeb}
Let $C\leq {}_n2\,\X(F)\otimes\cZ^1(F,\mu_n)$ be the subgroup generated by elements $\theta\otimes t^{1/n}$
such that $\theta$ and $t^{1/n}$ are matched, and define 
\[{}_n2\,\X(F)\wedge\cZ^1(F,\mu_n):=\frac{{}_n2\,\X(F)\otimes\cZ^1(F,\mu_n)}{C}\]
Write $\theta\wedge t^{1/n}$ in place of $\theta\otimes t^{1/n}+C$.
\end{enumerate}
\end{Definition}

The definition of matched pairs $(\theta,t^{1/n})$ is independent of the basis used in \eqref{eqns'},
since \eqref{eqns'} is equivalent to $(\theta,t)=0$ by Corollary~\ref{tt'}.

\begin{Theorem}\label{cocycleskim}
Assume \eqref{setup}, with $n$ a power of a prime $\ell$.
The cup product map ${}_n2\,\X(F)\otimes\cZ^1(F,\mu_n)\lr{}_n2\,\Br(F)$ has kernel $C$, and
induces a natural isomorphism
\[{}_n2\,\X(F)\wedge\cZ^1(F,\mu_n)\overset{\sim}\lr{}_n2\,\Br(F)\]
\end{Theorem}

\begin{proof}
Let $\un\chi\times\un x$ be a basis for ${}_n2\,\X(F)\times\cZ^1(F,\mu_n)$ 
of Corollary~\ref{tt'}, so $|\chi_i|=m'$ for $i\geq 1$.
The map ${}_n2\,\Br(F)\lr{}_n2\,\X(F)\otimes\cZ^1(F,\mu_n)$ defined by 
$(\chi_i,x_j)\longmapsto\chi_i\otimes x_j^{1/n}$ for $0\leq i<j$ 
obviously splits the cup product map.
On the other hand, if $0\leq i<j$ we have
\begin{align*}
\chi_j\otimes x_i^{1/n}
&=(\chi_j\otimes x_i^{1/n}+(x_i)_{m'}^*\otimes x_j^{1/n})-(x_i)_{m'}^*\otimes x_j^{1/n}\\
&=[(\chi_j+(x_i)_{m'}^*)\otimes x_i^{1/n}x_j^{1/n}-(x_i)_{m'}^*\otimes x_i^{1/n}-\chi_j\otimes x_j^{1/n}]
-(x_i)_{m'}^*\otimes x_j^{1/n}
\end{align*} 
Since each character is in ${}_n2\,\X(F)$,
the square-bracketed term is in $C$ by Corollary~\ref{tt'}.
Since $\un\chi\otimes\un x$ is a basis,
it follows that any element of ${}_n2\,\X(F)\otimes\cZ^1(F,\mu_n)$ can be expressed as a sum of
an element of $C$ and an element in the image of the splitting map.
Since $C\subset\ker(\cup)$ by Definition~\ref{wedge} and Corollary~\ref{tt'},
we must have $C=\ker(\cup)$, so we have the natural isomorphism.
\end{proof}

We next generalize the basis of Example~\ref{fb} and Corollary~\ref{tt'}.

\begin{Definition}\label{mbasis}
A basis $\un\theta\times\un t$ for ${}_n2\,\X(F)\times\cZ^1(F,\mu_n)$ is {\it matched} if it consists
of matched pairs $(\theta_i,t_i^{1/n})$, as in Definition~\ref{wedge}\eqref{wedgea}
(over only $i=0$ if $m'=1$), and if $m'\neq n$,
$\un t$ is in {\it standard form}, meaning
$\{t_1^{1/n},\dots,t_d^{1/n}\}$ is a uniformizer basis for $F^{\times 1/n}/F^\times$.
\end{Definition}

\begin{Remark}\label{newbegin}
If $\un\theta\times\un t$ is a matched basis for ${}_n2\,\X(F)\times\cZ^1(F,\mu_n)$, then 
by Corollary~\ref{tt'} and Definition~\ref{wedge}\eqref{wedgea}, 
the pairs $(\theta_i,t_i^{1/n})$ satisfy
the hypotheses (\ref{fba}$'$) and (\ref{fbb}$'$) of Corollary~\ref{tt'}, and $\un\theta\times\un t$
may then be used to define matched elements, by Corollary~\ref{tt'}. 
In particular, any matched basis may be used to define $C$ in Definition~\ref{wedge}\eqref{wedgeb}.
\end{Remark}

\begin{Proposition}\label{wedgebasis}
Assume \eqref{setup}.
A matched basis $\un\chi\times\un x$ induces a basis $\un\chi\wedge\un x$
for ${}_n2\,\X(F)\wedge\cZ^1(F,\mu_n)$, and a basis $\un\chi\ccup\un x$
for ${}_n2\,\Br(F)$.
\end{Proposition}

\begin{proof}
If $m'=1$ then $C=\langle\chi_0\otimes x_0^{1/n}\rangle$, and the $\chi_0\wedge x_j^{1/n}$ 
and $(\chi_0,x_j)$ for $j\geq 1$ clearly
form a basis for ${}_n2\,\X(F)\wedge\cZ^1(F,\mu_n)$ and ${}_n2\,\Br(F)$.
Assume $m'\neq 1$. Let 
\[\boldsymbol c=\{\chi_i\otimes x_i^{1/n},
(x_j)_{m'}^*\otimes x_k^{1/n}+\chi_k\otimes x_j^{1/n}\,:\; 0\leq i\leq d, 0\leq j<k\leq d\}
\]
Since $(x_j)_{m'}^*\otimes x_k^{1/n}+\chi_k\otimes x_j^{1/n}=
((x_j)_{m'}^*+\chi_k)\otimes x_j^{1/n}x_k^{1/n}-(x_j)_{m'}^*\otimes x_j^{1/n}-\chi_k\otimes x_k^{1/n}$, 
$\br{\boldsymbol c}\leq C$. The elements of $\boldsymbol c$ are independent:
Every element of ${}_n2\,\X(F)\otimes\cZ^1(F,\mu_n)$ has
a unique expression of the form $\prod_{j=0}^d\xi_j\otimes x_j^{1/n}$,
for characters $\xi_j$, and
a dependence relation
\[\sum_{i=0}^d a_i\chi_i\otimes x_i^{1/n}
+\sum_{0\leq i<j\leq d}a_{ij}((x_i)_{m'}^*\otimes x_j^{1/n}+\chi_j\otimes x_i^{1/n})=0\]
has $x_0^{1/n}$ term
$\xi_0=a_0\chi_0+\sum_{j=1}^d a_{0j}\chi_j$
and for $j\geq 1$, $ x_j^{1/n}$ term 
\[\xi_j=a_j\chi_j+\sum_{0\leq i<j}a_{ij}(x_i)_{m'}^*+\sum_{0\leq j<i\leq d}a_{ji}\chi_i\]
Since $\un x$ is a basis, each $\xi_j$ is zero.
Therefore since $\un\chi$ is a basis for ${}_n2\,\X(F)$, and each $\chi_i$ appears only once in each $\xi_j$, 
the dependence relation is trivial. Since $|\br{\boldsymbol c}|=nm^{\prime d(d+3)/2}$,
we conclude $\br{\boldsymbol c}=C$, so $\boldsymbol c$ is a basis for $C$.
Let $\b=\{\chi_i\otimes x_j^{1/n}\,:\;0\leq i<j\leq d\}$.
Then $\b\,\cup\,\boldsymbol c$ is another basis for ${}_n2\,\X(F)\otimes\cZ^1(F,\mu_n)$,
which shows $\b+C=\un\chi\wedge\un x$ is a basis for the cokernel,
hence that $\un\chi\ccup\un x$ is a basis for ${}_n2\,\Br(F)$.
\end{proof}

\section{Basis Change}

To use exterior algebra machinery to find minimal expressions for a given class in ${}_n2\,\Br(F)$,
we need to characterize basis change matrices on ${}_n2\,\X(F)\wedge\cZ^1(F,\mu_n)$.
To this end we have the following lemma, which characterizes integer matrices 
with well-defined actions on finite abelian groups 
that have unequal invariant factors, such as ${}_n2\,\X(F)$. It also lets us transition
to coefficients in $\tfrac 1n\Z/\Z$, which we use to define
the Hasse invariant.

\begin{Lemma}\label{omit}
Let $G=\Z/d_1\times\cdots\times\Z/d_r$ and
$\prestar G=\tfrac 1{d_1}\Z/\Z\oplus\cdots\oplus\tfrac 1{d_r}\Z/\Z$ be abelian groups with invariant factors
$d_i$, and let $d_{ij}=\gcd(d_i,d_j)$.
\begin{enumerate}[\rm(a)]
\item\label{omita}
Each $\rho\in\Aut(G)$ is representable by a $P=(p_{ij})\in\M_r(\Z)$
that satisfies $\tfrac{d_i}{d_{ij}}|\, p_{ij}$.
\item\label{omitb}
Each $\prestar\rho\in\Aut(\prestar G)$ is representable by a $\prestar P=(p_{ij})\in\M_r(\Z)$
that satisfies $\tfrac{d_j}{d_{ij}}|\, p_{ij}$.
\item\label{omitd}
Any $P\in\Aut(G)_\Z$ or $\prestar P\in\Aut(\prestar G)_\Z$ determines a $\rho$ or $\prestar\rho$
as in \eqref{omita} or \eqref{omitb}, and
$P=(p_{ij})$ and $P'=(p_{ij}')$ determine the same $\rho$ if and only
if $p_{ij}\equiv p_{ij}'\pmod{d_i}$. 
\item\label{omitc}
Let $\Aut(G)_\Z$ and $\Aut(\prestar G)$ denote the semigroups of such matrices. 
There is a natural bijection
\begin{align*}
\Aut(G)_\Z&\longleftrightarrow\Aut(\prestar G)_\Z\\
P=(p_{ij})&\longmapsto \prestar P=(\prestar p_{ij})=(\tfrac{d_j}{d_i}p_{ij})\\
Q^*=(q_{ij}^*)=(\tfrac{d_i}{d_j}q_{ij})&\longmapsfrom Q=(q_{ij})
\end{align*}
making a commutative diagram
\[\begin{tikzcd}
&G\arrow[r, "P"]\arrow[d, "\wr"]&G\arrow[d, "\wr"]\\
&\prestar G\arrow[r, "\prestar P"']&\prestar G
\end{tikzcd}\]
\end{enumerate}
\end{Lemma}

\begin{proof}
See also \cite{HR07}.
It is clear that any $\rho\in\Aut(G)$ is representable with respect to the standard basis
by an integer matrix $(p_{ij})$, whose $ij$-th entry maps $\Z/d_j$ to $\Z/d_i$, which is well-defined
if and only if $p_{ij}$
is divisible by $d_i/d_{ij}$. Similarly $\prestar\rho$ is representable by a $(\prestar p_{ij})$,
and $\prestar p_{ij}:\tfrac 1{d_j}\Z/\Z\to\tfrac 1{d_i}\Z/\Z$ is well defined if and only if $\prestar p_{ij}$
is divisible by $d_j/d_{ij}$. 
Two matrices $(p_{ij})$ and $(p_{ij}')$ determine the same $\rho$ if and only
if $[p_{ij}]_{d_i}=[p_{ij}']_{d_i}$ for each $i,j$, if and only if $p_{ij}\equiv p_{ij}'\pmod{d_i}$.
This proves \eqref{omita}, \eqref{omitb}, and \eqref{omitd}, and
\eqref{omitc} follows from the commutative diagram
\[\begin{tikzcd}
&\Z/d_j\arrow[r, "\frac{d_i}{d_{ij}}"]\arrow[d, "\wr"]&\Z/d_i\arrow[d, "\wr"]\\
&\tfrac 1{d_j}\Z/\Z\arrow[r, "\frac{d_j}{d_{ij}}"']&\tfrac 1{d_i}\Z/\Z
\end{tikzcd}\]
\end{proof}

\subsubsection{Character Bases}
Assume \eqref{setup}, with $n$ a power of $\ell$.
By Theorem~\ref{chargp}, ${}_n2\,\X(F)\isom\Z/n$ if $m'=1$, and
${}_n2\,\X(F)\isom\Z/n\times(\Z/m')^d$ if $m'\neq 1$.
When $m'\neq n$, a given basis for ${}_n2\,\X(F)$ is in standard form as in
Definition~\ref{mbasis}, i.e.,
ordered with the order-$n$ element first.
If $\un\theta$ and $\un\chi$ are bases, there is a basis change $[\id]_{\un\chi}^{\un\theta}$,
and by Lemma~\ref{omit} it is given by a matrix $R=(r_{ij})\in\Aut({}_n2\,\X(F))_\Z$.
If $m'=1$, $R=[r_{00}]$ with $r_{00}$ prime-to-$\ell$. 
If $m'\neq 1$ then $\tfrac n{m'}|\, r_{0j}$ for $j\geq 1$. 
If $m'\neq n$, $r_{00}$ and $R_{00}$ are invertible $\pmod\ell$, 
where $R_{00}$ is the submatrix obtained by deleting row and column $0$. 

\subsection{Matched Basis Change}\label{mmchange}
If $\un\chi\times\un x$ is a matched basis, then by Remark~\ref{newbegin},
$\un\theta\times\un t$ is a matched basis if and only if
$\un\theta\times\un t=\un\chi\times\un x(R\oplus S)$ for some $R=(r_{ij})$ and $S=(s_{ij})$
satisfying
\begin{equation}\label{mxeqns}
r_{0j}s_{ij}\equiv\tfrac n{m'}s_{0j}r_{ij}\pmod n\quad\text{and}\quad r_{ij}s_{kj}\equiv r_{kj}s_{ij}\pmod{m'}
\quad (i\geq 1, j\geq 0)
\end{equation}  
and the normalization equations of Definition~\ref{wedge}\eqref{wedgea},
\begin{align*}
r_{00}\equiv s_{00}\pmod {m'},\quad
&r_{0j}\equiv \tfrac n{m'} s_{0j}\pmod n,\;\text{ for $j\geq 1$}\\
\tfrac n{m'} r_{i0}\equiv s_{i0}\pmod {m'},\quad
&r_{ij}\equiv s_{ij}\pmod {m'},\;\text{ for $i,j\geq 1$}
\end{align*}
If $m'=1$, $R=[r_{00}]$ is invertible, hence \eqref{mxeqns} implies $s_{i0}\equiv 0\pmod n$.
We call such $R\oplus S\in\Aut({}_n2\,\X(F))_\Z\oplus\Aut(\cZ^1(F,\mu_n))_\Z$
{\it matched basis changes}, and write $R\oplus S=[\id]_{\un\theta\times\un t}^{\un\chi\times\un x}$.

If $P\in\Aut(\cZ^1(F,\mu_n))_\Z$ preserves standard form,
and $\tfrac n{m'}|\, p_{i0}$, then $P^*\oplus P$ is a matched basis change matrix, since
$p_{i0}^*=\tfrac{m'}np_{i0}$ and $p_{0j}^*=\tfrac n{m'}p_{0j}$ for $i,j\geq 1$, and $p_{ij}^*=p_{ij}$ otherwise.
This type will suffice below, though it is not the most general kind.

\begin{Proposition}\label{wedgebasischange}
Assume \eqref{setup}, with $n$ a power of $\ell$.     
A matched basis change $P^*\oplus P=[\id]_{\un\theta\times\un t}^{\un\chi\times\un x}$
induces on ${}_n2\,\X(F)\wedge\cZ^1(F,\mu_n)$ the basis change
$P^{*\wedge 2}=[\id]_{\un\theta\wedge\un t}^{\un\chi\wedge\un x}$.
\end{Proposition}

\begin{proof}
We ignore the ``zero'' basis elements of ${}_n2\,\X(F)$ when $m'=1$, 
but keep the degree-$d+1$ matrix $P^*$ in order to define $P^{*\wedge 2}$.
The matched bases induce bases $\un\theta\wedge\un t$ and
$\un\chi\wedge\un x$ on ${}_n2\,\X(F)\wedge\cZ^1(F,\mu_n)$ by Proposition~\ref{wedgebasis}.
Let $P=(p_{ij})$.
Since
$p^*_{0j}=\tfrac n{m'} p_{0j}$ and $p^*_{jl}= p_{jl}$ for $j,l\geq 1$
by Lemma~\ref{omit}, we compute for $l\geq 1$,
\begin{align*}
\theta_0\wedge t_l^{1/n}
&=\sum_{i=0}p^*_{i0}\chi_i\wedge \prod_{j=0}x_j^{p_{jl}/n}
 =\sum_{i\neq j}p^*_{i0}p_{jl}\chi_i\wedge x_j^{1/n}\\
&=\sum_{1\leq j}(p^*_{00}p_{jl}-p^*_{j0}p_{0l}\tfrac n{m'})\chi_0\wedge x_j^{1/n}
+\sum_{1\leq i<j}(p^*_{i0}p_{jl}-p^*_{j0}p_{il})\chi_i\wedge x_j^{1/n}\\
&=\sum_{1\leq j}(p^*_{00}p^*_{jl}-p^*_{j0}p^*_{0l})\chi_0\wedge x_j^{1/n}
+\sum_{1\leq i<j}(p^*_{i0}p^*_{jl}-p^*_{j0}p^*_{il})\chi_i\wedge x_j^{1/n}
\end{align*}
When ${m'}=1$ the only nonzero entries are 
$p^*_{00}p_{jl}^*\,\chi_0\wedge x_j^{1/n}$.
For $1\leq k<l$, when $m'\neq 1$,
\begin{align*}
\theta_k\wedge t_l^{1/n}
&=\sum_{1\leq j}(p^*_{0k}p_{jl}-p^*_{jk}p_{0l}\tfrac n{m'})\chi_0\wedge x_j^{1/n}
+\sum_{1\leq i<j}(p^*_{ik}p_{jl}-p^*_{jk}p_{il})\chi_i\wedge x_j^{1/n}\\
&=\sum_{1\leq j}(p^*_{0k}p^*_{jl}-p^*_{jk}p^*_{0l})\chi_0\wedge x_j^{1/n}
+\sum_{1\leq i<j}(p^*_{ik}p^*_{jl}-p^*_{jk}p^*_{il})\chi_i\wedge x_j^{1/n}
\end{align*}
Thus the basis change is $P^{*\wedge 2}$.
If $Q$ and $P$ give the same matched basis change, 
then $Q\equiv P\pmod n$, hence $Q^{*\wedge 2}$ and $P^{*\wedge 2}$ produce the same
basis change on ${}_n2\,\X(F)\times\cZ^1(F,\mu_n)$. Therefore the transformation
is well-defined, and this completes the proof.
\end{proof}

\section{Computing in $2\,\Br(F)$ with Skew-Symmetric Matrices}\label{skewsection}

By Theorem~\ref{cocycleskim}, for any $n$ not divisible by $p$ we have a natural isomorphism
\[{}_n2\,\X(F)\wedge\cZ^1(F,\mu_n)\overset{\sim}\lr{}_n2\,\Br(F)\]
The bases of Proposition~\ref{wedgebasis} determine 
for each element a skew-symmetric matrix,
which we will use to compute the index and decomposition into cyclic division algebras
of each Brauer class.

\begin{Theorem}\label{alta}
Assume \eqref{setup}, with $n$ a power of $\ell$.
A matched basis $\un\chi\times\un x$
determines an injective homomorphism
\[\alt_{\un\chi\times\un x}:{}_n2\,\Br(F)\lr\Alt_{d+1}(\tfrac 1n\Z/\Z)\]
defined by $[(\chi_i,x_j)]_{\un\chi\times\un x}=\tfrac 1{|\chi_i|}(e_{ij}-e_{ji})$
for $0\leq i<j\leq d$.
A matched basis change $P^*\oplus P=[\id]_{\un\chi\times\un x}^{\un\theta\times\un t}$ 
induces a commutative diagram
\[
\begin{tikzcd}
&\Alt_{d+1}(\tfrac 1n\Z/\Z)\arrow[dd,"P\,-\,P^\tr"]\\
{}_n2\,\Br(F)\arrow[ur,"{\alt_{\un\chi\times\un x}}"]\arrow[dr,"{\alt_{\un\theta\times\un t}}"']\\
&\Alt_{d+1}(\tfrac 1n\Z/\Z)
\end{tikzcd}
\]
Conversely, congruence transformation by 
any $P\in\M_{d+1}(\Z)$ satisfying $\tfrac n{m'}|\, p_{i0}$ for $i\geq 1$,
and such that $P$ is invertible $\pmod\ell$, and $p_{00}$ and $P_{00}$ are invertible $\pmod \ell$ when $m'\neq n$,
is induced by the matched basis change matrix 
$P^*\oplus P=[\id]_{\un\chi\times\un x}^{\un\theta\times\un t}\in\Aut({}_n2\,\X(F))_\Z\oplus\Aut(\cZ^1(F,\mu_n))_\Z$.
\end{Theorem}

\begin{proof}
A basis $\un\chi\wedge\un x$ for ${}_n2\,\X(F)\wedge\cZ^1(F,\mu_n)$ determines a map
\[\begin{tikzcd}
{}_n2\,\X(F)\wedge\cZ^1(F,\mu_n)\arrow[r]&\Alt_{d+1}(\tfrac 1n\Z/\Z)
\end{tikzcd}\]
by $\chi_i\wedge x_j^{1/n}\longmapsto\tfrac 1{|\chi_i|}(e_{ij}-e_{ji})$.
The basis change $P^*\oplus P=[\id]_{\un\chi\times\un x}^{\un\theta\times\un t}$ induces a
basis change $P^{*\wedge 2}=[\id]_{\un x^{\wedge 2}}^{\un t^{\wedge 2}}$ by Proposition~\ref{wedgebasischange},
which on $\Alt_{d+1}(\Z/n)$ is congruence transformation by $P^*$, but on $\Alt_{d+1}(\tfrac 1n\Z/\Z)$
is congruence transformation by $P=\prestar P^*$, as per Lemma~\ref{omit}.
Composing with the inverse of the natural isomorphism of Theorem~\ref{cocycleskim}
yields the result.

Conversely, congruence transformation by any $P$ satisfying the stated conditions 
is induced by the matched basis change $P^*\oplus P$ by \eqref{mmchange},
Proposition~\ref{wedgebasischange}, and what was just proved.
\end{proof}

\begin{Remark}
When $F$ is a (1-dimensional) local field, and $n$ is prime-to-$p$,
identifying $\Alt_2(\tfrac 1n\Z/\Z)$ with $\tfrac 1n\Z/\Z$ yields the (canonical)
Hasse invariant map
\[\inv_F:{}_n\Br(F)\isim\tfrac 1n\Z/\Z\]
It is canonical because the map determined by the Frobenius character $\chi_0$ and any uniformizer
give the same map into $\Alt_2(\tfrac 1n\Z/\Z)$: Any basis change matrix $P^*$ for ${}_n2\,\X(F)$
that preserves $\chi_0$ and preserves the value $1\pmod n$ of the uniformizer has the form
\[P^*=\bmx{1&p_{01}^*\\0&1}\pmod n\]
hence $P^{*\wedge 2}=[1]$. This shows the map ${}_n2\,\Br(F)\to\Alt_2(\tfrac 1n\Z/\Z)\isom\tfrac 1n\Z/\Z$ is independent of 
the basis used to define it.
\end{Remark}

We next prove the main result, that the skew-symmetric matrix assigned to a Brauer class 
computes its index, via the square root of the determinant, 
and the decomposition of its associated division algebra into cyclic division algebras, via a symplectic basis.

\subsection{Pfaffian}
The determinant of a skew-symmetric
matrix over a commutative ring is zero if the matrix 
has odd degree, and it is a square if the matrix has even degree.  
The positive square root is called the {\it pfaffian}.
Computations over $\Q/\Z$ can be well-defined as follows.

A degree-$r$ {\it skew-symmetric submatrix} $S$ of a skew-symmetrix matrix $A$
is the submatrix obtained by intersecting some set of $r$ rows $i_1,\dots,i_r$
with the columns $i_1,\dots,i_r$.

\begin{Definition}\label{pfaffian}
Suppose $A\in\Alt_{d+1}(\tfrac 1n\Z/\Z)$.
\begin{enumerate}[(a)]
\item
The {\it row subgroup} $\row(A)$ is the subgroup of $(\tfrac 1n\Z/\Z)^{d+1}$ generated by $A$'s rows.
\item
The {\it pfaffian subgroup}
$\pfaff(A)$ is the subgroup of $\Q/\Z$
generated by the pfaffians of all even-degree skew-symmetric submatrices
of $A$, computed na\"\i vely.
\end{enumerate}
\end{Definition}

This is \cite[Definition 2.7]{Br01b}.
``Computed na\"\i vely'' means computed from an arbitrary 
representative of $A$ in $\M_{d+1}(\Q)$, then interpreted (mod $\Z$).
For a given skew-symmetric $A$, $\pfaff(A)$ is well defined by cofactor expansion, 
and is invariant under congruence transformation by \cite[Proposition 2.8, Lemma 2.9]{Br01b}.
The isomorphism class of $\row(A)$ is also preserved, since congruence transformations
induce automorphisms of $(\tfrac 1n\Z/\Z)^{d+1}$.

Since $\row(A)$ and $\pfaff(A)$ are invariant under congruence transformation,
they can be computed from a 2-block matrix congruent to $A$.
Thus if
\[A\;\sim\;\sum_{i=0}^{\lfloor \tfrac{d-1}2\rfloor} a_{2i\,2i+1}(e_{2i\,2i+1}-e_{2i+1\,2i})\]
for $a_{2i\,2i+1}\in\Q/\Z$
and $d_i=|a_{2i\,2i+1}|$ is the order of the $i$-th block, 
then 
\begin{align*}
\row(A)&\isom\prod_{i=0}^r\tfrac 1{d_i}\Z/\Z\times\tfrac 1{d_i}\Z/\Z\\
\pfaff(A)&=\br{\tfrac 1{d_1 d_2\cdots d_r}}
\end{align*}
where $r=\lfloor \tfrac{d-1}2\rfloor$.
The numbers $d_0,d_0,d_1,d_1,\dots,d_r,d_r$, where $d_i\geq 1$, 
are the invariant factors of the finite abelian group $\row(A)$.

\begin{Theorem}\label{alt}
Assume \eqref{setup}, with $n$ a power of $\ell$.
Let $\un\chi\times\un x$ be a matched basis, with respect to which
$\alpha\in{}_n2\,\Br(F)$ has skew-symmetric matrix $\alt_{\un\chi\times\un x}(\alpha)=A\in\Alt_{d+1}(\tfrac 1n\Z/\Z)$.
Then
\[\ind(\alpha)=|\pfaff(A)|\]
If $d_0,d_0,\dots,d_r,d_r$ are the invariant factors of $\row(A)$, where $r=\lfloor\tfrac{d-1}2\rfloor$,
then there exists a matched basis $\un\theta\times\un t$
for ${}_n2\,\X(F)\times\cZ^1(F,\mu_n)$ with respect to which
\[\alpha=\tfrac n{d_0}(\theta_0,t_1)+(t_2,t_3)_{d_1}+\cdots+(t_{2r},t_{2r+1})_{d_r}\] 
and $\alpha$'s division algebra decomposes into cyclic division algebras
\[\Delta(\alpha)=\Delta(\tfrac n{d_0}(\theta_0,t_1))\otimes_F\Delta((t_2,t_3)_{d_1})
\otimes_F\cdots\otimes_F\Delta((t_{2r},t_{2r+1})_{d_r})\] 
\end{Theorem}

\begin{proof}
We may assume without loss of generality that $\alpha\in{}_n2\,\Br(F)$ has order $n$. 

Suppose $m'=1$, so $m$ divides $2$. Since $\chi_0\in{}_n2\X(F)$, 
$\chi_0$ is unramified by Theorem~\ref{chargp}.
By Theorem~\ref{Brgp}, $\alpha=\sum_{j=1}^d(\chi_0,x_j^{a_j})$, for some integers $a_j$,
and $\alt_{\un\chi\times\un x}(\alpha)=\sum_{j=1}^d a_{0j}(e_{0j}-e_{j0})$, where $a_{0j}=a_j/n$.
Since $|\alpha|=n$, at least one of the $a_j$ is a unit $\pmod n$.
Therefore $t_1^{1/n}=x_1^{a_1/n}\cdots x_d^{a_d/n}$ is part of a uniformizer basis
via a basis change on $\cZ^1(F,\mu_n)$ of the form $P=[1]\oplus P_{00}$, 
hence part of a matched basis $\{\theta_0\}\times\un t$ obtained by $P^*\oplus P$.
Then $\alpha=(\theta_0,t_1)$,
and the corresponding central simple algebra $\Delta(\theta_0,t_1)$
is a division algebra, since $t_1^e$ is a norm from $F(\theta_0)$
if and only if $\tfrac en\v(t_1)\in\Gamma_{F(\theta_0)}=\Gamma_F$, if and only if $n|e$
(see Reiner \cite[Theorem 14.1, p.143]{Reiner}).
The resulting congruence transformation on $\Alt_{d+1}(\tfrac 1n\Z/\Z)$ is $P\,-\,P^\tr$
by Theorem~\ref{alta}, so that $\alt_{\un\theta\times\un t}(\alpha)=PAP^\tr=\tfrac 1n(e_{01}-e_{10}).$ 
Since there is only one $2$-block in $\alt_{\un\theta\times\un t}(\alpha)$,
the row subgroup $\row(A)$ is isomorphic to $\tfrac 1n\Z/\Z\times\tfrac 1n\Z/\Z$, and
the pfaffian subgroup $\pfaff(A)$ is $\br{\tfrac 1n}\leq\Q/\Z$,
so $|\pfaff(A)|=n$. This completes the proof when $m'=1$.

Assume $m'\neq 1$.
Any matched basis transforms into one with unramified character $\chi_0$ using a matrix of form $P^*\oplus P$,
defining the first column of $P^*$ to give an unramified $\chi_0$, using the identity for the other elements,
and setting $P=\prestar P^*$. Since the resulting congruence transformation
on $\Alt_{d+1}(\tfrac 1n\Z/\Z)$ in Theorem~\ref{alta}
does not affect $\row(A)$ or $\pfaff(A)$, we may assume $\chi_0$ is unramified. 
We now construct a new basis $\un\theta\wedge\un t$, with respect to which $\alpha$
has a ``linked 2-block form'', and such that $\theta_0=\chi_0$. 
Let
\begin{enumerate}[(a)]
\item
$C_{ij}(c)=I+c e_{ij}$ for $0\leq i\neq j\leq d$
\item
$D_j(u)=I+(u-1)e_{jj}$ for a unit $u$, and $0\leq j\leq d$
\item
$E_{ij}=I-e_{ii}-e_{jj}+e_{ij}+e_{ji}$ for $0\leq i\neq j\leq d$
\end{enumerate}
When 
$P=C_{ij}(c)=I+c e_{ij}$ for $i\neq j$, 
$PAP^\tr$ replaces the $i$-th row and column with $\row(i)+c\cdot\row(j)$ and $\col(i)+c\cdot\col(j)$.
When $P=D_j(u)$, $PAP^\tr$ multiplies the $j$-th row and column by $u$.
When $P=E_{ij}$, $PAP^\tr$ permutes the row/column $i$ and row/column $j$.

By ``apply $P$\,'' we mean ``apply the congruence transformation $P\,\cdot\, P^\tr$ ''.

\noindent{\bf I.}
Write $A=(a_{ij})$, where as usual $0\leq i,j\leq d$.
If any of the $a_{0j}$ are nonzero for $j\geq 1$, apply $C_{1j}(1)$
as necessary so that $|a_{01}|$ has maximal order among the $a_{0j}$. 
Then apply $D_1(u)$, if necessary, so that $a_{01}=1/|a_{01}|$.
For each $j:a_{0j}\neq 0$ and $1<j$, apply $C_{j1}(b)$ as necessary, so that
$a_{0j}=0$ for $1<j$.
The first row and column of the transformed matrix $A$
now have at most one nonzero entry, at locations $(0,1)$ and $(1,0)$.
The basis change matrix has the form $P=[1]\oplus B$
for some $n$-invertible $B\in\M_d(\Z)$.

\noindent{\bf II.}
Let $c=1$. 
Repeat for as long as $c\leq d-1$ and $a_{ij}\neq 0$ for some $i,j:c\leq i<j\leq d$:
\begin{enumerate}[\rm a)]
\item
Suppose $a_{i_0j_0}$ has maximal order among all $a_{ij}$ with $c\leq i<j\leq d$.
If $c<i_0$, apply $C_{c i_0}(1)$ as necessary so that $a_{c\,j_0}$
has this maximal order.
Note this operation does not affect $a_{k\,k+1}$ for $k<c$.
If $c+1<j_0$, apply $C_{c+1\,j_0}(1)$ as necessary to that
$a_{c\,c+1}$ has maximal order among the $a_{ij}$
with $i,j:c\leq i<j\leq d$.
Then apply $D_{c+1}(u)$ for a unit $u$(mod $n$) so
that $a_{c\,c+1}=1/|a_{c\,c+1}|$.
The composite matrix $P$ for these operations is again of the form $[1]\oplus B$.
\item
For each $j:c+1<j\leq d$ such that $a_{cj}\neq 0$, 
apply $C_{j\,c+1}(b)$ as necessary so that $a_{cj}=0$, for all $j:c+1<j\leq d$.
Row $c$ now contains two nonzero entries, $a_{c\,c+1}$ and $a_{c\,c-1}$.
The composite $P$ is again of the form $[1]\oplus B$.
\item
Increase $c$ by $1$.
\end{enumerate}
If $P$ is the composite of all of the above transformations, $P$ has form $[1]\oplus B$, 
and $A'=PAP^\tr$ has the ``linked $2$-block form''
\[A'=\sum_{i=0}^{d-1} a_{i\,i+1}'(e_{i\,i+1}-e_{i+1\,i})\]
with $|a_{i\,i+1}'|>|a_{i+1\,i+1}'|$ for $i\geq 1$, and 
$a_{i\,i+1}'=1/|a_{i\,i+1}'|$.
By Theorem~\ref{alta}, $P$ is induced by the matched basis change 
$P^*\oplus P=[\id]_{\un\chi\times\un x}^{\un\chi'\times\un x'}$,
so that $A'=\alt_{\un x'}(\alpha)$, and
$\alpha=\tfrac n{m_1}(\chi_0',x_1')+(x_1',x_2')_{m_2}+\cdots+(x_{d-1}',x_d')_{m_d}$
where $m_{i+1}|m_i$ for $i\geq 2$.
Since $P$ has the form $[1]\oplus B$, the first column of $P^*$ is $e_1$,
so $\chi_0'=\chi_0$ is unramified.

To conserve notation,
assume $A$ is in the linked 2-block form, with respect to the matched basis 
$\un\chi\times\un x$, and
\begin{equation}\label{linked}
\alpha=\tfrac n{m_1}(\chi_0,x_1)+(x_1,x_2)_{m_2}+\cdots+(x_{d-1},x_d)_{m_d}
\end{equation}
where $m_{i+1}|\, m_i$ for $i\geq 2$, $\chi_0$ is unramified, and $m_i\geq 1$.
Let $D$ be $\alpha$'s $F$-division algebra, and let $S$ and $T$ be the division algebras
defined by
\[[S]=\tfrac n{m_1}(\chi_0,x_1)\,,\quad [T]=(x_1,x_2)_{m_2}+(x_2,x_3)_{m_3}+\cdots+(x_{d-1},x_d)_{m_d}\]
Then $S$ is nicely semiramified and $T$ is totally ramified, since $\un x$ is in standard form,
and $D\sim S\otimes_F T$.
Since $\chi_0$ is unramified and $x_1^{1/n}$ is a uniformizer, 
$S$ is isomorphic to $\Delta(\tfrac n{m_1}\chi_0,x_1)$. 
Thus $\Gamma_S/\Gamma_F=\langle\tfrac 1{m_1}\v(x_1)\rangle+\Gamma_F$.
As in Section~\ref{value}, we will compute $\ind(\alpha)$ by the formula
\[[D:F]=[\Gamma_D:\Gamma_F][\Gamma_T:\Gamma_F]\]

The linked 2-block form of a skew-symmetric matrix for $[T]$ is put in disjoint 2-block form
by applying $C_{31}({m_2/m_3}),C_{53}(m_4/m_5),\cdots,C_{2q_0+1\,2q_0-1}(m_{2q_0}/m_{2q_0+1})$ in succession,
where $q_0=\lceil \tfrac d2\rceil-1$, and we index the rows and column starting at $i=1$ instead of 
$i=0$. 
The composite basis change matrix $P_{00}\in\M_d(\Z)$
determines a uniformizer basis $\{t_1^{1/n},\dots,t_d^{1/n}\}$, by 
$t_{2i-1}=x_{2i-1}x_{2i+1}^{-m_{2i}/m_{2i+1}}$, and $t_{2i}=x_{2i}$,
with respect to which
\[[T]=(t_1,t_2)_{m_2}+(t_3,t_4)_{m_4}+\cdots+(t_{2r_0-1},t_{2r_0})_{m_{2r_0}}\]
where $r_0=\lfloor \tfrac d2 \rfloor$.
The central simple algebra corresponding to this expression is well-known 
to be a division algebra, since $\{t_1^{1/n},\dots,t_d^{1/n}\}$ is a uniformizer basis,
cf. Example~\ref{max}. Therefore it is isomorphic to $T$.
We compute 
\[\Gamma_T/\Gamma_F=\br{\tfrac 1{m_{2i}}\v(t_{2i-1}), \tfrac 1{m_{2i}}\v(t_{2i}):1\leq i\leq\lfloor \tfrac d2\rfloor}+\Gamma_F\]
hence $[\Gamma_T:\Gamma_F]=\prod_{i=1}^{\lfloor \tfrac d2\rfloor} m_{2i}^2$, and
since $\Gamma_D=\Gamma_S+\Gamma_T$, 
\[\Gamma_D/\Gamma_F=\br{\tfrac 1{m_1}\v(x_1),
\tfrac 1{m_{2i}}\v(t_{2i-1}), \tfrac 1{m_{2i}}\v(t_{2i}):1\leq i\leq\lfloor \tfrac d2\rfloor}+\Gamma_F\]

To compute $\ind(\alpha)$ it remains to compute $[\Gamma_D:\Gamma_F]$.
The expression of $x_1^{1/n}$ in terms of the $t_i^{1/n}$ is given by the first column of the
composite basis change matrix $P_{00}\in\M_d(\Z)$ above. It is easy to show by induction that
the first column is $e_1+\tfrac{m_2}{m_3}e_3+\tfrac{m_2 m_4}{m_3 m_5}e_5+\cdots+
\tfrac{m_2 m_4\cdots m_{2q_0}}{m_3 m_5\cdots m_{2q_0+1}}e_{2q_0+1}$, 
where $q_0=\lceil \tfrac d2\rceil-1=\lfloor\tfrac{d-1}2\rfloor$.
Therefore
\[\Gamma_D/\Gamma_F=\bigg\langle\sum_{i=0}^{\lceil \tfrac d2\rceil-1} \tfrac{m_2 m_4\cdots m_{2i}}{m_1 m_3\cdots m_{2i+1}}\v(t_{2i+1}),
\tfrac 1{m_{2i}}\v(t_{2i-1}), \tfrac 1{m_{2i}}\v(t_{2i}):1\leq i\leq\lfloor \tfrac d2\rfloor\bigg\rangle+\Gamma_F\]
We replace the term
$(m_2\cdots m_{2i}/m_1 \cdots m_{2i+1})\v(t_{2i+1})$ by $0$ if it is in $\Gamma_F$.

Let $R\in\M_{d\,d+1}(\Z)$ be the relations matrix, defined by the exact sequence
\[\Z^{d+1}\overset{R}{\lr}\Z^d\lr\Gamma_D/\Gamma_F\lr 0\]
For example, if $d=5$ and $m_1 m_3 m_5>m_2 m_4$, then
\[R\;=\;\bmx{m_1 &m_2&0&0&0&0\\0&0&m_2&0&0&0\\ \tfrac{m_1 m_3}{m_2}&0&0&m_4&0&0\\
0&0&0&0&m_4&0\\ \tfrac{m_1 m_3 m_5}{m_2 m_4}&0&0&0&0&1}\]
and if $d=6$ and $m_1 m_3 m_5>m_2 m_4$,
\[R\;=\;\bmx{m_1&m_2&0&0&0&0&0\\0&0&m_2&0&0&0&0\\  \tfrac{m_1 m_3}{m_2}&0&0&m_4&0&0&0\\
0&0&0&0&m_4&0&0\\  \tfrac{m_1 m_3 m_5}{m_2 m_4}&0&0&0&0&m_6&0\\0&0&0&0&0&0&m_6}\]
By e.g. \cite[Ch. VII, Section 2, Appendix]{Hung}, the order of $\Gamma_D/\Gamma_F$
is the largest minor of $R$.
We have $m_{i+1}|m_i$ for $i\geq 2$ by construction.
Thus $m_3 m_5\cdots m_j<m_2 m_4\cdots m_{j-1}$ for all $j$,
and the entries in the $0$-th column of $R$ are nonincreasing as the row position increases,
and if $m_1 m_3\cdots m_{j_0}/m_2 m_4\cdots m_{j_0-1}$ is greater than one for some $j_0$, 
then for $j<j_0$, $m_1 m_3\cdots m_j/m_2 m_4\cdots m_{j-1}$ is also greater than one.
Thus the $0$-th column of $R$ has alternating nonzero entries until some row, past which
it is all zeros.

Since at most one column has more than one nonzero entry, 
cofactor expansion shows every minor consists of a single term.
Furthermore, by the Smith Normal Form Theorem, since $R$ has rank $d$,
a minor of degree $c<d$ cannot be larger than all degree-$d$ minors.
Thus $[\Gamma_D:\Gamma_F]$ is the supremum of the degree-$d$ minors, 
which are each defined by a single deleted column.
Deleting column $0$ yields 
\[M_0:=m_2^2 m_4^2\cdots m_{2\lfloor \tfrac d2\rfloor}^2\]
Let $j_0$ be the largest odd number such that $m_1 m_3\cdots m_{j_0}>m_2 m_4\cdots m_{j_0-1}$.
Then the entry in row $j_0-1$ is the bottom nonzero entry of column $0$.
Deleting column $j_0$ deletes $m_{j_0+1}$, or $1$ if $d$ is odd and $j_0=d$. 
This yields the minor 
\begin{align*}
M_{j_0}:=&\tfrac{m_1 m_3\cdots m_{j_0}}{m_2 m_4\cdots m_{j_0-1}} 
m_2^2 m_4^2\cdots m_{j_0-1}^2 m_{j_0+1} m_{j_0+3}^2 m_{j_0+5}^2\cdots m_{2\lfloor \tfrac d2\rfloor}^2\\
&=m_1 m_2 m_3\cdots m_{j_0+1}m_{j_0+3}^2 m_{j_0+5}^2\cdots m_{2\lfloor \tfrac d2\rfloor}^2
\end{align*}
If $j<j_0$ is also odd then $m_1 m_3\cdots m_j>m_2 m_4\cdots m_{j-1}$, and deleting column $j+1$ yields
the minor $M_j=m_1 m_2 m_3\cdots m_{j+1}m_{j+3}^2 m_{j+5}^2\cdots m_{2\lfloor \tfrac d2\rfloor}^2$.
We compute, using $m_{i+1}| m_i$ for $i\geq 2$,
\[\frac{M_{j_0}}{M_j}=\frac{m_{j+2}m_{j+3}\cdots m_{j_0}m_{j_0+1}}{m_{j+3}^2 m_{j+5}^2\cdots m_{j_0+1}^2}
=\frac{m_{j+2}m_{j+4}\cdots m_{j_0}}{m_{j+3}m_{j+5}\cdots m_{j_0+1}}\geq 1\]
Therefore $M_{j_0}$ is largest of the minors obtained by deleting column $j$ for odd $j\leq j_0$.
If $j$ is odd and greater than $j_0$, then deleting column $j$ makes row $j-1$ equal to zero,
and the resulting minor is zero.
If $j$ is even and greater than zero, deleting column $j$ removes the only nonzero entry in row $j-1$,
and the resulting minor is zero.
We conclude
\[|\Gamma_D/\Gamma_F|=\max\{M_0,M_{j_0}\}\]

Since $|\Gamma_T/\Gamma_F|=M_0$, and $[D:F]=[\Gamma_D:\Gamma_T][\Gamma_D:\Gamma_F]$,
\begin{align*}
[D:F]&=\max\{M_0,M_{j_0}^2/M_0\}\\
&=\max\{m_2^2 m_4^2\cdots m_{2\lfloor \tfrac d2\rfloor}^2,
m_1^2 m_3^2\cdots m_{j_0}^2 m_{j_0+3}^2 m_{j_0+5}^2\cdots m_{2\lfloor \tfrac d2\rfloor}^2\}
\end{align*}
Suppose that $j_0<2\lfloor \tfrac d2\rfloor$. 
Then by definition of $j_0$ we have
\[m_1 m_3\cdots m_{j_0}m_{j_0+2}<m_2 m_4\cdots m_{j_0-1}m_{j_0+1}\]
hence \[m_1^2 m_3^2\cdots m_{j_0}^2 m_{j_0+2}^2 m_{j_0+3}^2\cdots m_{2\lfloor \tfrac d2\rfloor}^2
<m_2^2 m_4^2\cdots m_{2\lfloor \tfrac d2\rfloor}^2\]
Thus $[D:F]=M_0$ in this case.
Also in this case, if $d$ is even then since $m_{i+1}|m_i$ for $i\geq 2$,
we have 
\[m_1^2 m_3^2\cdots m_{j_0+2}^2 m_{j_0+4}^2\cdots m_{d-1}^2
<m_2^2 m_4^2\cdots m_{j_0+1}^2 m_{j_0+3}^2\cdots m_{d-2}^2 m_d^2\]
and if $d$ is odd then similarly
\[m_1^2 m_3^2\cdots m_{j_0+2}^2 m_{j_0+4}^2\cdots m_d^2< m_2^2 m_4^2\cdots m_{j_0+1}^2 m_{j_0+3}^2
\cdots m_{d-1}^2\]
On the other hand,
if $d$ is even and $j_0=d-1$, $[D:F]=\max\{m_2^2 m_4^2\cdots m_d^2,m_1^2 m_3^2\cdots m_{d-1}^2\}$, and
if $d$ is odd and $j_0=d$, 
$[D:F]=\max\{m_2^2 m_4^2\cdots m_{d-1}^2,m_1^2 m_3^2\cdots m_d^2\}$.
We conclude that in any case
\[[D:F]=
\begin{cases}
\max\{m_2^2 m_4^2\cdots m_d^2,m_1^2 m_3^2\cdots m_{d-1}^2\}&\text{ if $d$ is even}\\
\max\{m_2^2 m_4^2\cdots m_{d-1}^2,m_1^2 m_3^2\cdots m_d^2\}&\text{ if $d$ is odd}
\end{cases}
\]

The pfaffian subgroup of $\alt_{\un\chi\times\un x}(\alpha)$ is easily computed from the linked 2-block form of $\alpha$
in \eqref{linked}. We obtain
\[\pfaff(\alt_{\un\chi\times\un x}(\alpha))=\begin{cases}
\br{\tfrac 1{m_1 m_3 \cdots m_{d-1}},\tfrac 1{m_2 m_4\cdots m_d}}&\text{ if $d$ is even}\\[4pt]
\br{\tfrac 1{m_1 m_3\cdots m_{d-2}},\tfrac 1{m_2 m_4\cdots m_{d-1}}}&\text{ if $d$ is odd}
\end{cases}\]
Thus $\ind(\alpha)=|\pfaff(\alt_{\un\chi\times\un x}(\alpha))|$, as desired.

We next show there exists a matched basis $\un\theta\times\un t$ such that 
$\alt_{\un\theta\times\un t}(\alpha)$ is in disjoint $2$-block form.
If $m'=n$ then ${}_n\X(F)$ and $\cZ^1(F,\mu_n)$ are isomorphic by Kummer theory,
and any basis change $P\oplus P$ on ${}_n\X(F)\times\cZ^1(F,\mu_n)$ preserves the subgroup $C$
of Definition~\ref{wedge}, hence induces a basis change $P^{\wedge 2}$
on ${}_n\X(F)\wedge\cZ^1(F,\mu_n)$.
The algorithm for putting an arbitrary skew-symmetric matrix into disjoint 2-block form with such a matrix
is \cite[Lemma 1.10]{Br01b}, adapted from \cite[Section 6.2]{Jac:BA1}, and the result follows.

If $m'\neq n$, then since $\alpha$ has order $n$, the term $\tfrac n{m_1}(\chi_0,x_1)$ 
in the linked 2-block form \eqref{linked} has order $n$, 
so $m_1$ is divisible by $m_2$, hence $m_{i+1}|\,m_i$ for all $i$. 
Applying 
\[P=C_{2q_0\,2q_0-2}(\tfrac{m_{2q_0-1}}{m_{2q_0}})\cdots C_{42}(\tfrac{m_3}{m_4})C_{20}(\tfrac{m_1}{m_2})\]
where $q_0$ is at most $\lfloor \tfrac d2\rfloor$, then yields the desired $2$-block form,
leaving at most $2\lfloor\tfrac{d-1}2\rfloor+1$ nonzero factors. 
To check that $P$ comes from a matched basis change as in Theorem~\ref{alta}, 
it suffices to check that each factor
satisfies the conditions in that theorem. 
This is immediate for all but possibly $C_{20}(m_1/m_2)$.
But since $m_2$ divides $m'$, and $m_1=n$, $n/m'$ divides the entries of the $0$-th column
of $C_{20}({m_1/m_2})$ below the top, so this too satisfies the definition, and the claim is proved.

We have shown that there exists a matched basis $\un\theta\times\un t$ with respect to which
$\alt_{\un\theta\times\un t}(\alpha)$ is in disjoint $2$-block form. The corresponding expression
for $\alpha$ is 
\[\alpha=\tfrac n{m_1}(\theta_0,t_1)+(t_2,t_3)_{m_3}+\cdots+(t_{2r},t_{2r+1})_{m_{2r+1}}\]
where $r=\lfloor\tfrac{d-1}2\rfloor$.
The corresponding tensor product of central simple algebras has degree $\pfaff(A)$,
which we have shown equals $\ind(\alpha)$. Therefore 
\[\Delta(\tfrac n{m_1}\theta_0,t_1)\otimes_F \Delta(t_2,t_3)_{m_3}\otimes_F\cdots\otimes_F \Delta(t_{2r},t_{2r+1})_{m_{2r+1}}\]
is a division algebra.
This completes the proof.
\end{proof}

The next result generalizes Example~\ref{max} to any matched basis.

\begin{Corollary}\label{bigdivalg}
Assume \eqref{setup}, with $n$ a power of a prime $\ell$, 
and either $\ell$ is odd or $\ell=2$ and $q\equiv 1\pmod{2n}$.
Let $\un\theta\times\un t$ be any matched basis.
Then 
\[D\;=\;\Delta(\theta_0,t_1)\otimes_F \Delta(t_2,t_3)_{m'}\otimes_F\cdots\otimes_F \Delta(t_{2r},t_{2r+1})_{m'}\]
where $r=\lfloor\frac{d-1}2\rfloor$, is a division algebra of period $n$ and index $nm^r$,
with $r+1=\lfloor\frac{d+1}2\rfloor$ cyclic tensor factors.
\end{Corollary}

\begin{proof}
We have $\per(D)=n$ since $\theta_0$ has order $n$,
and $\ind(D)=nm^r$ by the pfaffian formula. Since $nm^r=\deg(D)$,
$D$ is a division algebra, which is a tensor product of $r+1$ cyclic factors.
\end{proof}

\begin{Example}
The disjoint 2-block form may not retain the unramified character.
Suppose $m=\ell^3$ for $\ell$ an odd prime, $x_0$ is defined over $k$, and 
\[\alpha=(x_0,x_1)_{\ell^3}+(x_1,x_2)_{\ell^2}+(x_2,x_3)_\ell\]
Then
\[\alt_{\un x}(\alpha)=
\bmxr{0&1/\ell^3&0&0\\-1/\ell^3&0&1/\ell^2&0\\0&-1/\ell^2&0&1/\ell\\0&0&-1/\ell&0}\] 
Applying $P=C_{20}(\ell)$ gives a disjoint $2$-block form: Defining $\un t$
by $\un x=\un t P$
yields $t_0=x_0 x_2^{-\ell}$, $x_1=t_1$, $x_2=t_2$, $x_3=t_3$,
and 
\[\alpha=(t_0,t_1)_{\ell^3}+(t_2,t_3)_\ell=(x_0 x_2^{-\ell},x_1)_{\ell^3}+(x_2,x_3)_\ell\]
The new leading character $(t_0)_{\ell^3}^*$ is ramified, making the index computation
$(\ell^4)$ nontrivial.
\end{Example}

\subsection{Pfaffian Formulas}
Here are the first few pfaffian formulas.
Let $A=(a_{ij})\in\Alt_{d+1}(\Q/\Z)$ and let $\pfaff_{d+1}(A)$ 
denote the pfaffian.
Then
\begin{align*}
\pfaff_2(A)&=a_{01}\\
\pfaff_4(A)&=a_{01}a_{23}-a_{02}a_{13}+a_{03}a_{12}\\
\pfaff_6(A)&=a_{01} a_{23} a_{45} - a_{01} a_{24} a_{35} + a_{01} a_{25} a_{34}
- a_{02} a_{13} a_{45} + a_{02} a_{14} a_{35} - a_{02} a_{15} a_{34}
\\ &\quad
+ a_{03} a_{12} a_{45} - a_{03} a_{14} a_{25} + a_{03} a_{15} a_{24}
- a_{04} a_{12} a_{35} + a_{04} a_{13} a_{25} - a_{04} a_{15} a_{23}
\\ &\quad
+ a_{05} a_{12} a_{34} - a_{05} a_{13} a_{24} + a_{05} a_{14} a_{23}.
\end{align*}
Thus if $\alt(\alpha)=A=(a_{ij})\in\Alt_{d+1}(\Q/\Z)$, then
\begin{enumerate}[\rm (a)]
\item
If $d=1$, $\per(\alpha)=\ind(\alpha)=|a_{01}|$.
\item
If $d=2$, $\per(\alpha)=\ind(\alpha)=\lcm[|a_{01}|,|a_{02}|,|a_{12}|]$.
\item
If $d=3$, $\per(\alpha)=\lcm[|a_{ij}|]$
and
$\ind(\alpha)=\lcm[\per(\alpha),|a_{01}a_{23}-a_{02}a_{13}+a_{03}a_{12}|]$.
\item
If $d=4$, $\per(\alpha)=\lcm\{|a_{ij}|\}$ and
\begin{align*}
\ind(\alpha)=&\lcm[\per(\alpha),|a_{01}a_{23}-a_{02}a_{13}+a_{03}a_{12}|,
|a_{01}a_{24}-a_{02}a_{14}+a_{04}a_{12}|,\\ &|a_{01}a_{34}-a_{03}a_{14}+a_{04}a_{13}|,
|a_{02}a_{34}-a_{03}a_{24}+a_{04}a_{23}|,
|a_{12}a_{34}-a_{13}a_{24}+a_{14}a_{23}|]
\end{align*}
\end{enumerate}
The $d=3$ case is the first with unequal period and index.

\begin{Example}
Suppose $F=\F_5((x_1))((x_2))((x_3))((x_4))((x_5))$, the $5$-dimensional local field,
$x_0=[2]_5$, and we have a sum of quaternions
\begin{align*}
\alpha=&\;(2,x_1)+(2,x_2)+(2,x_3)+(2,x_4)+(2,x_5)
+(x_1,x_2)+(x_1,x_3)+(x_1,x_4)+(x_1,x_5)\\
&+(x_2,x_3)+(x_2,x_4)+(x_2,x_5)+(x_3,x_4)+(x_3,x_5)
+(x_4,x_5)
\end{align*}
The matrix is
\[A=\bmxr{
0&\tfrac 12&\tfrac 12&\tfrac 12&\tfrac 12&\tfrac 12\\[4pt]
-\tfrac 12&0&\tfrac 12&\tfrac 12&\tfrac 12&\tfrac 12\\[4pt]
-\tfrac 12&-\tfrac 12&0&\tfrac 12&\tfrac 12&\tfrac 12\\[4pt]
-\tfrac 12&-\tfrac 12&-\tfrac 12&0&\tfrac 12&\tfrac 12\\[4pt]
-\tfrac 12&-\tfrac 12&-\tfrac 12&-\tfrac 12&0&\tfrac 12\\[4pt]
-\tfrac 12&-\tfrac 12&-\tfrac 12&-\tfrac 12&-\tfrac 12&0\\[4pt]}
\]
By the formulas above, $\pfaff_2(A)=1/2$,
$\pfaff_4(A)=1/4$, and, since $\pfaff_6$ is an alternating sum of an odd number of 
fractions $1/8$, $\pfaff_6(A)=1/8$. Therefore
$\per(\alpha)=2$ and $\ind(\alpha)=\lcm\{2,4,8\}=8$.

Similarly, since the pfaffian always has an odd number of terms, 
if $F=\F_p((x_1))\cdots((x_d))$, $x_0$ generates $\F_p^\times$,
and \[\alpha=\sum_{0\leq i<j\leq d}(x_i,x_j)_{m'}\]
then $\per(\alpha)=m'$, and $\ind(\alpha)=m^{\prime r}$, where $r=\lfloor\tfrac{d+1}2\rfloor$.
\end{Example}

\section{Failure in the Case $\ell=2$}

\begin{Example}\label{failure}
We illustrate the failure of a Brauer class of $2$-power order to behave like
an alternating form, even when $F$ contains $\mu_n$. Let $F=\F_p((x_1))((x_2))((x_3))$,
and let $m=n=2^{\v_2(p-1)}$.
Let $\un x$ be a basis for $\cZ^1(F,\mu_n)$, in standard form.
Let
\[\alpha=(x_1,x_2)_n+(x_1,x_3)_n+(x_2,x_3)_n\]
then 
\[\alt_{\un x}(\alpha)=A=\bmxr{0&0&0&0\\[2pt]0&0&\tfrac 1n&\tfrac 1n\\[2pt]0&-\tfrac 1n&0&\tfrac 1n\\[2pt]
0&-\tfrac 1n&-\tfrac 1n&0}\]
The pfaffian subgroup is generated by the subpfaffians of degree $1$, so $\pfaff(A)=\br{1/n}$,
predicting an index of $n$.
Let
\[P=\bmxr{1&0&0&0\\0&1&-1&1\\0&0&1&0\\0&0&0&1}\]
be a basis change, and define $\un t$ by
$\un x=\un t P$.
By substitution,
\begin{align*}
\alpha&=(t_1,t_1^{-1}t_2)_n+(t_1,t_1t_3)_n+(t_1^{-1}t_2,t_1t_3)_n\\
&=(t_1,t_1^{-1})_n+(t_1,t_2)_n+(t_1,t_1)_n+(t_1,t_3)_n+(t_1^{-1}t_1)_n
+(t_1^{-1},t_3)_n+(t_2,t_1)_n+(t_2,t_3)_n\\
&=(t_1,t_1)_n+(t_2,t_3)_n=(-1,t_1)_n+(t_2,t_3)_n
\end{align*}
Since $n=2^{\v_2(p-1)}$, $(-1,t_1)_n$ has order $2$ by Theorem~\ref{tt}, 
and since the totally ramified class $(t_2,t_3)_n$ has value group disjoint from the
semiramified class $(-1,t_1)_n$, we compute $\ind(\alpha)=2n$ using the methods of Theorem~\ref{alt}.
If ${}_n\Br(F)$ behaved like a group of skew-symmetric matrices, functorial with respect to basis change, 
the new matrix would be
\[\alt_{\un t}(\alpha)=\bmxr{0&\frac 12&0&0\\-\frac 12&0&0&0\\0&0&0&\frac 1n\\0&0&-\frac 1n&0}\]
But we compute instead
\[PAP^\tr=\bmxr{0&0&0&0\\0&0&0&0\\0&0&0&\frac 1n\\0&0&-\frac 1n&0}\]
which is consistent, of course, with the prediction of index $n$ from the pfaffian.

Summarizing, though the expression for ${}_n\Br(F)$ determines an {\it ad hoc} map into the group of
skew-symmetric matrices, this map appears not to be functorial with respect to basis change
on classes in ${}_n\Br(F)-{}_n2\,\Br(F)$.
This would defeat the
reason for defining the map in the first place, since it means we can't use matrix diagonalization to find a minimal
representation for the class, and cannot produce an index formula. 
\end{Example}

\section{Obstruction when ${}_n\Br(F)\neq{}_n2\,\Br(F)$}\label{naturalmaps}

Assume \eqref{setup}, with $n$ a power of a prime $\ell$, $\mu_n\leq F^\times$, and $G=\Gal(F^{1/n}/F)$.
Since $\mu_n\leq F^ \times$, we substitute $\tfrac 1n\Z/\Z$ for $\mu_n$ using $\zeta_n^*$.
Consider the diagram
\[\begin{tikzcd}& &\H^2(F,\tfrac 1n\Z/\Z)\arrow[dr, dashed, "\alt?"]&\\
&\H^2(G,\tfrac 1n\Z/\Z)\arrow[ur, "\inf"]\arrow[rr,"{[\alt]}"']
& &\Alt(G,\tfrac 1n\Z/\Z)\\
\end{tikzcd}\] 
The bottom arrow is a split surjection by Lemma~\ref{uct}.
We now show $[\alt]$ factors through $\H^2(F,\tfrac 1n\Z/\Z)$ if and only if $q\equiv 1\pmod{n^2}$, i.e.,
$\mu_{n^2}\leq F^\times$.
Then we show that a natural map $\alt:{}_n\Br(F)\lr\Alt(G,\tfrac 1n\Z/\Z)$
exists if and only if $n\neq 2^{\v_2(q-1)}$, as suggested by Example~\ref{failure}.
In particular, we always have a map on ${}_n2\,\Br(F)$.

\begin{Proposition}\label{tg}
Assume \eqref{setup}, with $m=n$ a power of $\ell$. 
Let $L=F^{1/n}$, and put $G=\Gal(L/F)$.
Then we have an exact sequence
\[\begin{tikzcd}
0\arrow[r]&\H^1(L,\tfrac 1n\Z/\Z)\arrow[r,"\tg"]&\H^2(G,\tfrac 1n\Z/\Z)
\arrow[r,"\inf"]&\H^2(F,\tfrac 1n\Z/\Z)\arrow[r]&0
\end{tikzcd}\]
Let $[\alt]:\H^2(G,\tfrac 1n\Z/\Z)\lr\Alt(G,\tfrac 1n\Z/\Z)$ be the map of Lemma~\ref{uct}, 
let $\un\phi=\{\phi_0,\phi_1,\dots,\phi_d\}$ be the image in $G$
of $\un\Phi=\{\Phi_0,\Phi_1,\dots,\Phi_d\}$ of \eqref{origpres} under the projection,
and let $\xi$ be an element of $\H^1(L,\tfrac 1n\Z/\Z)$.
Then for $i\leq j$,
\[[\alt](\tg(\xi))(\phi_i,\phi_j)=\begin{cases}
-\frac{q-1}n\xi(\Phi_j^n)&\text{if $i=0<j$}\\
0&\text{otherwise}
\end{cases}\]
In particular, $[\alt]$ factors through
${}_n\Br(F)$ if and only if $q\equiv 1\pmod{n^2}$.
\end{Proposition}

\begin{proof}
Since $\mu_n\leq F^\times$, $L$ contains all cyclic extensions of degree $n$ by Kummer theory,
and $G\isom(\Z/n)^{d+1}$ by Theorem~\ref{chargp}.
All elements of $\H^1(F,\tfrac 1n\Z/\Z)$ and $\H^2(F,\tfrac 1n\Z/\Z)$ are split by $L$, by Theorem~\ref{chargp}
and Theorem~\ref{Brgp}, respectively.
Therefore $\H^1(G,\tfrac 1n\Z/\Z)=\H^1(F,\tfrac 1n\Z/\Z)$, and
$\ker[\H^2(F,\tfrac 1n\Z/\Z)\to\H^2(L,\tfrac 1n\Z/\Z)^G]=\H^2(F,\tfrac 1n\Z/\Z)$.
The 7-term sequence \eqref{7term} applied to $1\to\G_L^\ttr\to\G_F^\ttr\to G\to 1$ then yields
\[\begin{tikzcd}
0\arrow[r]&\H^1(L,\tfrac 1n\Z/\Z)^G\arrow[r,"\tg"]&\H^2(G,\tfrac 1n\Z/\Z)\arrow[r,"\inf"]&\H^2(F,\tfrac 1n\Z/\Z)
\end{tikzcd}\]
The cup product $\H^1(F,\tfrac 1n\Z/\Z)\otimes\H^1(F,\tfrac 1n\Z/\Z)\to\H^2(F,\tfrac 1n\Z/\Z)$ is onto
by Theorem~\ref{Brgp}, and since $\H^1(F,\tfrac 1n\Z/\Z)=\H^1(G,\tfrac 1n\Z/\Z)$, this map factors through
$\H^2(G,\tfrac 1n\Z/\Z)$. Therefore the map on the right is surjective, as desired.

We claim the action of $G$ on $\H^1(L,\tfrac 1n\Z/\Z)$ is trivial.
Suppose $\xi\in\H^1(L,\tfrac 1n\Z/\Z)$, then the action
is given by $\xi^{\phi_i}=\xi\circ\Phi_i^{-1}$, where $\Phi_i^{-1}$ acts on $\G_L^\ttr$
by conjugation.
Since $G$ is isomorphic to $(\Z/n)^{d+1}$, and $G$'s basis $\un\phi$ is the image of $\un\Phi$, 
a general element $\prod_{i=0}^d\Phi_i^{a_i}\in\G_F^\ttr$ is in $\G_L^\ttr$ if and only if
$n|\, a_i$ for all $i$, and
by \eqref{origpres}, $\G_L^\ttr$ has presentation
\[\G_L^\ttr=\langle\un\Phi^n:[\Phi_0^n,\Phi_j^n]=\Phi_j^{n(q^n-1)},\;
[\Phi_i^n,\Phi_j^n]=e,\;\forall i,j\geq 1\rangle\]
Since $\un\Phi^n$ is a basis for $\G_L^\ttr$,
to prove the claim it suffices to show $\xi(\Phi_i^{-1}\Phi_j^n\Phi_i)=\xi(\Phi_j^n)$
for $i,j:0\leq i,j\leq d$, for any $\xi\in\H^1(L,\tfrac 1n\Z/\Z)$.
By \eqref{origpres}, 
\[\Phi_i^{-1}\Phi_j^n\Phi_i=\begin{cases}
\Phi_0^n&\text{if $i=j=0$}\\
\Phi_j^{nq^{-1}}&\text{if $i=0$ and $j>0$}\\
\Phi_i^{q^n-1}\Phi_0^n&\text{if $i>0$ and $j=0$}\\
\Phi_j^n&\text{if $i,j>0$}
\end{cases}\]
Since $n\xi=0$ and $q\equiv 1\pmod n$, 
we have $\xi(\Phi_j^{nq^{-1}})=q^{-1}\xi(\Phi_j^n)=\xi(\Phi_j^n)$,
and 
\[\xi(\Phi_i^{q^n-1}\Phi_0^n)=(1+q+\cdots+q^{n-1})\xi(\Phi_i^{q-1})+\xi(\Phi_0^n)
=n\xi(\Phi_i^{q-1})+\xi(\Phi_0^n)=\xi(\Phi_0^n)\]
This proves the claim, 
finishing the construction of the exact sequence.

Let $s:\prod_{i=0}^d\phi_i^{a_i}\mapsto\prod_{i=0}^d\Phi_i^{a_i}$, $0\leq a_i\leq n-1$,
be a section of the map $\G_F^\ttr\to G$.
By \cite[Prop. 1.6.5]{NSW} and \cite[Thm. 2.1.7]{NSW},
$\tg(\xi)\in\cZ^2(G,\tfrac 1n\Z/\Z)$ is defined by
$\tg(\xi)(\sigma,\tau)=-\xi(s(\sigma\tau)^{-1}s(\sigma)s(\tau))$.
Since $[\alt](\tg(\xi))(\phi_i,\phi_j)=\tg(\xi)(\phi_i,\phi_j)-\tg(\xi)(\phi_j,\phi_i)$,
it remains to compute $\tg(\xi)(\phi_i,\phi_j)$ for each $i,j$.
Using the relations \eqref{origpres}, and the fact that $G$ is abelian, we get
\[s(\phi_i\phi_j)^{-1}s(\phi_i)s(\phi_j)=\begin{cases}
\Phi_j^{-1}\Phi_i^{-1}\Phi_i\Phi_j=e&\text{if $i\leq j$ or $j>0$}\\
\Phi_i^{-1}\Phi_0^{-1}\Phi_i\Phi_0=\Phi_i^{q^{-1}-1} &\text{if $i>j=0$}
\end{cases}\]
Thus $\tg(\xi)(\phi_i,\phi_j)=0$ if $i\leq j$ or $j>0$.
Since $q^{-1}-1=-q^{-1}(q-1)$, $|\xi|$ divides $n$, and $q\equiv q^{-1}\equiv 1\pmod n$, 
we have for $i>j=0$,
\[\tg(\xi)(\phi_i,\phi_0)=-\xi(\Phi_i^{-q^{-1}(q-1)})
=\xi(\Phi_i^{q-1})=\tfrac{q-1}n\xi(\Phi_i^n)\qquad (i>0)\]
Therefore
$[\alt](\tg(\xi))(\phi_0,\phi_i)=\tg(\xi)(\phi_0,\phi_i)-\tg(\xi)(\phi_i,\phi_0)=-\tfrac{q-1}n\xi(\Phi_i^n)$,
as claimed.

The map $[\alt]:\H^2(G,\tfrac 1n\Z/\Z)\to\Alt(G,\tfrac 1n\Z/\Z)$ factors through $\H^2(F,\tfrac 1n\Z/\Z)$ 
if and only if $\tg(\H^1(L,\tfrac 1n\Z/\Z))$ is contained in $\ker([\alt])$,
if and only if $n$ divides $\tfrac{q-1}n$, i.e., $q\equiv 1\pmod{n^2}$.
\end{proof}

\begin{Theorem}\label{unnatural}
Assume \eqref{setup}, with $m=n$ a power of a prime $\ell$.
Let $L=F^{1/n}$, and set $G=\Gal(L/F)\isom(\Z/n)^{d+1}$.
There exists a natural injective homomorphism
\[\alt:{}_n2\,\Br(F)\lr\Alt(G,\tfrac 1n\Z/\Z)\]
which does not extend to ${}_n\Br(F)$ when $n= 2^{\v_2(q-1)}$, i.e., when ${}_n\Br(F)\neq{}_n2\,\Br(F)$.
\end{Theorem}

\begin{proof}
We identify $\H^1(F,\tfrac 1n\Z/\Z)\otimes\cZ^1(F,\mu_n)$ 
with $\H^1(F,\tfrac 1n\Z/\Z)\otimes\H^1(F,\tfrac 1n\Z/\Z)$ using $\zeta_n^*$, and the
basis elements $\chi_i\otimes x_j^{1/n}$ then have form $\chi_i\otimes\chi_j$.
The subgroup $C\leq{}_n2\,\X(F)\otimes{}_n\X(F)$ of Theorem~\ref{cocycleskim} has basis 
\[\{\chi_i\otimes\chi_i,\tfrac{|\chi_j|}{m'}
(\chi_j\otimes\chi_k+\chi_k\otimes\chi_j),\;0\leq i\leq d,0\leq j<k\leq d\}
\]
by the proof of Proposition~\ref{wedgebasis}.
Application of $\alt$ shows $C$ is contained in the kernel $K$
of Lemma~\ref{uct}, hence the natural map
$\alt:{}_n2\,\Br(F)\to\Alt(G,\tfrac 1n\Z/\Z)$ exists.

We complete the proof with a lemma
that isolates a basis of elements of order $2$ in $\H^1(F,\tfrac 1n\Z/\Z)\otimes\H^1(F,\tfrac 1n\Z/\Z)$
that map to zero in $\Alt(G,\tfrac 1n\Z/\Z)$ but not in $\H^2(F,\tfrac 1n\Z/\Z)$,
foiling the definition of $\alt$ on $\H^2(F,\tfrac 1n\Z/\Z)$ when $n=2^{\v_2(q-1)}$.

\begin{Lemma}\label{neqker}
Assume \eqref{setup}, with  $n$ a power of $\ell$, $\mu_n\leq F^\times$,
and $G=\Gal(F^{1/n}/F)$.
Let $\un\phi=\{\phi_0,\dots,\phi_d\}$ be a basis for $G$,
with dual basis $\un\phi^*=\{\phi_0^*,\dots,\phi_d^*\}$
in $\H^1(F,\tfrac 1n\Z/\Z)$, and consider the cup product map
\[\ccup:\H^1(F,\tfrac 1n\Z/\Z)\otimes\H^1(F,\tfrac 1n\Z/\Z)\lr\H^2(G,\tfrac 1n\Z/\Z)\]
\begin{enumerate}[\rm (a)]
\item\label{tga}
If $\ell$ is odd, then
$\phi_i^*\ccup\phi_i^*=0$ in $\H^2(G,\tfrac 1n\Z/\Z)$, $\forall i:0\leq i\leq d$.
\item\label{tgb}
If $\ell=2$ then 
$\phi_i^*\ccup\phi_i^*$ has order $2$ in $\H^2(G,\tfrac 1n\Z/\Z)$.
\item\label{tgc}
If $\ell=2$ then
$\phi_i^*\ccup\phi_i^*=0$ in $\H^2(G,\tfrac 1{2n}\Z/\Z)$.
\item\label{tgd}
If $\ell=2$ then $\inf(\phi_i^*\ccup\phi_i^*)=0$ in $\H^2(F,\tfrac 1n\Z/\Z)$
if and only if $q\equiv 1\pmod{2e}$, where $e=[\Gamma_{F(\phi_i^*)}:\Gamma_F]$.
\end{enumerate}
\end{Lemma}

\begin{proof}
Since $\H^1(F,\tfrac 1n\Z/\Z)=\H^1(G,\tfrac 1n\Z/\Z)$, the cup product map makes sense.
For each $\sigma\in G$, define $a_\sigma:0\leq a_\sigma<n$ by $\phi_i(\sigma)=a_\sigma/n$.
Then for any $\sigma,\tau\in G$,
\[a_{\sigma\tau}=\begin{cases}
a_\sigma+a_\tau&\text{if $a_\sigma+a_\tau<n$}\\
a_\sigma+a_\tau-n&\text{if $a_\sigma+a_\tau\geq n$}
\end{cases}\]
Let $\Z[\Z]$ denote the group ring of the additive group of integers.
Let $\Z[G]$ be the (additive) $\Z[\Z]$-module
with the action $a\cdot b:=\phi_i^a b$ for all $a\in\Z$ and $b\in\Z[G]$.
Let $\gamma_{\phi_i}\in\cZ^1(\Z,\Z[G])$ be the 1-cocycle
defined by $\gamma_{\phi_i}(1)=1$.
Then $\gamma_{\phi_i}(0)=0$, and by the 1-cocycle condition,
\[
\gamma_{\phi_i}(a)=1+\phi_i+\cdots+\phi_i^{a-1}\qquad (a\in\N)\] 

We first prove \eqref{tga}.
Suppose $\ell$ is odd, and define $\beta\in\cC^1(G,\tfrac 1n\Z/\Z)$ by 
\[\beta(\sigma)=\phi_i^*(\gamma_{\phi_i}(a_\sigma))-(\phi_i^*\otimes\phi_i^*)(\sigma,\sigma)\]
where $\phi_i^*\otimes\phi_i^*$ is the bilinear form with values in $\tfrac 1n\Z/\Z$,
and we extend $\phi_i^*$ to $\Z[G]$ using the rule 
$\phi_i^*(\tau-\sigma)=\phi_i^*(\tau)-\phi_i^*(\sigma)$.
Since the action of $G$ on $\tfrac 1n\Z/\Z$ is trivial,
$\partial\beta(\sigma,\tau):=\beta(\sigma)+\beta(\tau)-\beta(\sigma\tau)$, hence
\begin{align*}
\partial\beta(\sigma,\tau)
&=\phi_i^*(\gamma_{\phi_i}(a_\sigma)+\gamma_{\phi_i}(a_\tau)-\gamma_{\phi_i}(a_{\sigma\tau}))
\\&\qquad
-(\phi_i^*\otimes\phi_i^*)(\sigma,\sigma)-(\phi_i^*\otimes\phi_i^*)(\tau,\tau)+
(\phi_i^*\otimes\phi_i^*)(\sigma\tau,\sigma\tau)\\
&=\phi_i^*(\gamma_{\phi_i}(a_\sigma)+\gamma_{\phi_i}(a_\tau)-\gamma_{\phi_i}(a_{\sigma\tau}))
+2\phi_i^*\otimes\phi_i^*(\sigma,\tau)\\
&=\begin{cases}
\phi_i^*(\gamma_{\phi_i}(a_\sigma)-\phi_i^{a_\tau}\gamma_{\phi_i}(a_\sigma))+2\frac{a_\sigma a_\tau}n
&\text{if $a_\sigma+a_\tau<n$}\\
\phi_i^*(\gamma_{\phi_i}(a_\sigma)-\phi_i^{a_\tau}\gamma_{\phi_i}(a_\sigma)-\gamma_{\phi_i}(n))+2\frac{a_\sigma a_\tau}n
&\text{if $a_\sigma+a_\tau\geq n$}\\
\end{cases}\\
&=\begin{cases}
\frac 1n+\cdots+\frac{a_\sigma-1}n
-(\frac{a_\tau}n+\frac{a_\tau+1}n+\cdots+\frac{a_\tau+a_\sigma-1}n)+2\frac{a_\sigma a_\tau}n
&\text{if $a_\sigma+a_\tau<n$}\\
\frac 1n+\cdots+\frac{a_\sigma-1}n
-(\frac{a_\tau}n+\frac{a_\tau+1}n+\cdots+\frac{a_\tau+a_\sigma-1}n)-\phi_i^*(\gamma_{\phi_i}(n))+2\frac{a_\sigma a_\tau}n
&\text{if $a_\sigma+a_\tau\geq n$}\\
\end{cases}\\
&=\begin{cases}
-\frac{a_\sigma a_\tau}n+2\frac{a_\sigma a_\tau}n
&\text{if $a_\sigma+a_\tau<n$}\\
-\frac{a_\sigma a_\tau}n+2\frac{a_\sigma a_\tau}n-\phi_i^*(\gamma_{\phi_i}(n))
&\text{if $a_\sigma+a_\tau\geq n$}\\
\end{cases}\\
&=\begin{cases}
\frac{a_\sigma a_\tau}n
&\text{if $a_\sigma+a_\tau<n$}\\
\frac{a_\sigma a_\tau}n-\frac{n(n-1)}{2n}
&\text{if $a_\sigma+a_\tau\geq n$}\\
\end{cases}
\end{align*}

If $\ell$ is odd, then $2|n-1$, so $n(n-1)/2n=0$,
and $\partial\beta=\phi_i^*\otimes\phi_i^*$.
Therefore $\phi_i^*\ccup\phi_i^*=0$ in $\H^2(G,\tfrac 1n\Z/\Z)$ when $\ell$ is odd, proving \eqref{tga}.

If $\ell=2$, then $2|n$, so $n(n-1)/2n=-1/2$.
The short exact sequence of trivial $G$-modules
\[0\lr\tfrac 1n\Z/\Z\lr\tfrac 1{2n}\Z/\Z\lr\tfrac 12\Z/\Z\lr 0\]
combined with the fact that $G$ has exponent $n$ yields the long exact sequence fragment
\[\begin{tikzcd}
0\arrow[r]&\H^1(G,\tfrac 12\Z/\Z)\arrow[r, "\delta"]&\H^2(G,\tfrac 1n\Z/\Z)\arrow[r]&\H^2(G,\tfrac 1{2n}\Z/\Z)
\end{tikzcd}\]
Since $G\isom(\Z/n)^{d+1}$ has basis $\{\phi_i:0\leq i\leq d\}$, 
$\H^1(G,\tfrac 12\Z/\Z)$ is generated by the $\tfrac n2\phi_i^*$,
and by construction 
\[\delta(\tfrac n2\phi_i^*)(\sigma,\tau)=
\frac{a_\sigma+a_\tau-a_{\sigma\tau}}{2n}=
\begin{cases}
0&\text{if $a_\sigma+a_\tau<n$}\\
-1/2&\text{if $a_\sigma+a_\tau\geq n$}
\end{cases}\]
Thus we find $\partial\beta=\phi_i^*\otimes\phi_i^*+\delta(\tfrac n2\phi_i^*)$.
Therefore $\phi_i^*\ccup\phi_i^*=\delta(\tfrac n2\phi_i^*)$ has order $2$ in $\H^2(G,\tfrac 1n\Z/\Z)$,
and is trivial in $\H^2(G,\tfrac 1{2n}\Z/\Z)$. In fact, $\delta(\tfrac n2\phi_i^*)=\partial\varepsilon$,
where $\varepsilon\in\cC^1(G,\tfrac 1{2n}\Z/\Z)$ is defined by $\varepsilon(\sigma)=a_\sigma/2n\in\tfrac 1{2n}\Z/\Z$.
This proves \eqref{tgb} and \eqref{tgc}.

Let $\phi_i^*=(x_i)_n^*$, where $(x_i)_n\in F^\times\!/n$.
When $\ell=2$ we have $(\phi_i^*,-x_i)=(\phi_i^*,-1)+(\phi_i^*,x_i)=0$ 
in $\H^2(F,\tfrac 1n\Z/\Z)$,
as mentioned in \eqref{cyclic}, 
and $(\phi_i^*,-1)=0$ if and only if $\phi_i^*$ extends to a character of order $2n$, 
by Albert's criterion.
Suppose $[\Gamma_{F(\phi_i^*)}:\Gamma_F]=e$.
By Theorem~\ref{chargp} we may write $\phi_i^*=\chi+\lambda$, where $\chi$ is unramified and $\lambda$
is totally ramified of order $e$. The character $\chi$ always extends, since it is defined over a
finite field, hence $\phi_i^*$ extends if and only if $\lambda$ extends, and by Kummer theory,
$\lambda$ extends to a character of order $2e$ if and only if $q\equiv 1\pmod{2e}$. 
Therefore $(\phi_i^*,x_i)$ equals zero if and only if $q\equiv 1\pmod{2e}$,
and since $(\phi_i^*,x_i)=\inf(\phi_i^*\ccup\phi_i^*)$, this proves \eqref{tgd}.
\end{proof}

\noindent
{\it Finish proof of Theorem~\ref{unnatural}.}
By Lemma~\ref{neqker}\eqref{tga} and \eqref{tgd},
the natural map 
\[\alt:{}_n\X(F)\otimes{}_n\X(F)\to\Alt(G,\tfrac 1n\Z/\Z)\] 
which factors through $\H^2(G,\tfrac 1n\Z/\Z)$, giving $[\alt]$,
factors through $\H^2(F,\tfrac 1n\Z/\Z)$ on precisely the subgroup ${}_n2\,\Br(F)$,
the elements not of the form $(t,t)_n$ when 
$(t)_n^*$ is totally ramified of order $n=2^{\v_2(q-1)}$. 
Since ${}_n\Br(F)={}_n2\,\Br(F)\Leftrightarrow n\neq 2^{\v_2(q-1)}$,
this completes the proof.
\end{proof}

\bibliographystyle{abbrv} 
\bibliography{hnx.bib}

\end{document}